\documentclass[11pt]{article}
\usepackage{amssymb,amsmath,bm}
\usepackage{amsthm}
\usepackage{xcolor}
\usepackage{textcomp}
\usepackage{enumerate}      
\usepackage{graphicx}        
\usepackage{caption, subcaption}

\usepackage{url}

\usepackage{tabu}

\usepackage[export]{adjustbox}

\usepackage{mathrsfs}

\usepackage{url}

\usepackage{array}
\newcommand{\PreserveBackslash}[1]{\let\temp=\\#1\let\\=\temp}
\newcolumntype{C}[1]{>{\PreserveBackslash\centering}p{#1}}
\newcolumntype{R}[1]{>{\PreserveBackslash\raggedleft}p{#1}}
\newcolumntype{L}[1]{>{\PreserveBackslash\raggedright}p{#1}}
\renewcommand{\setminus}{{\smallsetminus}}

\usepackage{amssymb,amsmath,bm}
\usepackage{amsthm}
\usepackage{hyperref}
\usepackage{mathrsfs}
\usepackage{textcomp}
\usepackage{enumerate}
\usepackage{graphicx}
\usepackage{tikz, tikz-cd}
\usepackage{verbatim}
\usetikzlibrary{decorations.markings}

\tikzset{->-/.style={decoration={
  markings,
  mark=at position #1 with {\arrow{>}}},postaction={decorate}}}
\tikzset{-<-/.style={decoration={
  markings,
  mark=at position #1 with {\arrow{<}}},postaction={decorate}}}

\newtheorem{theorem}{Theorem}[section]
\newtheorem{lemma}[theorem]{Lemma}
\newtheorem{proposition}[theorem]{Proposition}
\newtheorem{definition}[theorem]{Definition}
\newtheorem{corollary}[theorem]{Corollary}

\theoremstyle{remark}
\newtheorem{remark}[theorem]{Remark}

\theoremstyle{remark}

\numberwithin{equation}{section}

\begin{document}

\title{\bf Adjoint Reidemeister torsion from hyperbolic gluing equations}
\author{Shuang Ming and Baojun Wu}
\date{}

\maketitle
\begin{abstract}
Motivated by the study of asymptotics of quantum invariants, Dimofte and Garoufalidis introduced a power series associated to a suitable ideal triangulation of a cusped hyperbolic $3$-manifold. They proved that its constant term can be written in terms of Neumann-Zagier data and the complex shape parameters of ideal tetrahedra, and conjectured that it equals the adjoint twisted Reidemeister torsion. On the other hand, in the study of asymptotics of Turaev-Viro type invariants of cusped $ 3$-manifolds, the authors, together with Liu, Sun and Yang, found that the one-loop terms of their asymptotic expansions could be written in terms of Gram matrices and decorated edge lengths of ideal tetrahedra. In this paper, we prove the conjecture of Dimofte and Garoufalidis and relate the one-loop term appearing in the Turaev-Viro type invariant to the adjoint twisted Reidemeister torsion.
\end{abstract}

\section{Introduction}

The semiclassical study of quantum invariants of $3$-manifolds asks how geometric and topological information emerges from their combinatorial definitions. Volume conjectures concern the leading exponential behaviour~\cite{Kas95,CY18}; the one-loop contribution probes infinitesimal fluctuations around the corresponding classical geometry. The role of torsion in this description goes back to Schwarz's interpretation of analytic torsion in abelian field theory~\cite{Sch78,Sch79} and Witten's semiclassical analysis of Chern--Simons theory~\cite{Wit89}. These ideas motivate the expected appearance of adjoint Reidemeister torsion in quantum asymptotics. They do not, however, identify the one-loop expression produced by a finite-dimensional model with the torsion of the underlying $3$-manifold.

The one-loop conjecture of Dimofte and Garoufalidis~\cite{DG13} makes this identification precise for a model built from ideal tetrahedra. Its two sides have different origins. The one-loop invariant is expressed through the linearized gluing equations and a monomial in the shape parameters, whereas adjoint Reidemeister torsion is defined using the full based chain complex with coefficients in the adjoint local system. Although both constructions are combinatorial, equality between them does not follow merely from identifying their infinitesimal deformation spaces. One must also recover the determinants of the changes of basis and the normalization determined by peripheral homology.

A complementary motivation comes from the Turaev--Viro-type state integrals of Liu, Ming, Sun, Wu and Yang~\cite{LMSWY25a,LMSWY26}. These models use modular-double $b$-$6j$ symbols~\cite{PT01}, which are related to the fusion kernels of Virasoro conformal blocks~\cite{PT99}, together with their ideal variants in the cusped setting. Their geometric one-loop expressions raise a parallel question: how does adjoint Reidemeister torsion emerge from this construction, beyond the leading hyperbolic volume?
\par

In this paper, we give affirmative answers to both questions for cusped hyperbolic $3$-manifolds. We prove the Dimofte--Garoufalidis one-loop conjecture for $\rho_0$-regular ideal triangulations and, in the geometric case, identify the Turaev--Viro-type geometric one-loop expression with the appropriately normalized absolute square of adjoint Reidemeister torsion.

\subsection{Asymptotic expansion of quantum invariants and one-loop conjecture}
The quantum invariants are usually defined using combinatorial descriptions of $3$-manifolds. For instance, the Reshetikhin-Turaev invariant \cite{RT91} is defined using surgery diagrams of $3$-manifolds. The Turaev-Viro invariant \cite{TV92} and Kashaev invariants \cite{Kas95} are defined using (ideal) triangulations of $3$-manifolds. Hence the asymptotic study of these quantum invariants suggests that their asymptotic expansions are written in terms of combinatorial data together with parameters describing the geometric shape of the gluing pieces. Inspired by the study of asymptotics of quantum invariants, Dimofte and Garoufalidis constructed a perturbative series from enhanced Neumann–Zagier data \cite{DG13}. Its one-loop invariant is a determinant of the linearized gluing and peripheral equations multiplied by a monomial determined by a flattening. They established its independence of auxiliary choices and its invariance under suitable Pachner moves, and conjectured its equality with adjoint Reidemeister torsion.

The conjecture of Dimofte and Garoufalidis was initially proposed for hyperbolic knot complements and was later generalized to hyperbolic 3-manifolds with boundary homeomorphic to a union of tori by
Siejakowski \cite{Sie21} and Pandey-Wong \cite{PW23}. The conjecture has been established for several substantial families. Yoon \cite{Yoo24} proved a twisted version for hyperbolic once-punctured torus bundles. Dunfield, Garoufalidis and Yoon \cite{DGY23} subsequently proved a polynomial version for layered triangulations of fibered 3-manifolds with toroidal boundary. Pandey and Wong \cite{PW23} proved the conjecture for fundamental shadow link complements and manifolds obtained by sufficiently long Dehn fillings. Garoufalidis and Yoon \cite{GY26} proved it for hyperbolic two-bridge knots, using the earlier work of Ohtsuki and Takata together with an explicit computation. Besides verifying examples, Siejakowski \cite{Sie21} gave a cohomological interpretation of the infinitesimal gluing equations and their peripheral normalization. His torsion factorization recovers the determinant contribution and isolates the remaining identification with the flattening monomial.

{Our first theorem establishes this identity in the setting of $\rho_0$-regular ideal triangulations.}

\begin{theorem}\label{thm:main1}[One-loop formula]
Let M be a hyperbolic 3-manifold with toroidal boundary and let 
$$\boldsymbol{\gamma}=\{\gamma_{1},\ldots, \gamma_{|V|}|\gamma_{i}\subset \partial_{i}M\}$$ be a system of simple closed curves of $\partial M$. Let $\mathcal{T}$ be a $\rho_{0}$-regular ideal triangulation of $M$, i.e., the character $[\rho_{0}]$ is in the image of the pseudo-developing map $\mathcal{V}(\mathcal{T})\to X(M)$, where $\rho_{0}$ is the discrete faithful representation $\pi_{1}(M)\to PSL(2, \mathbb C)$. Then for any $\boldsymbol{z}=\{z_{\Delta}|\Delta\in T\}\in\mathcal{V}(\mathcal T)\subset \mathcal{Z}(\mathcal{T})$ such that the induced $PSL(2, \mathbb{C})$-representation $\rho(\boldsymbol{z})$ is $\boldsymbol{\gamma}$-regular, we have
\begin{equation}\label{eq:main1}
 \mathrm{Tor}(M, \rho, \boldsymbol\gamma)
 =
 \frac{1}{\det(C^TC)}
 \det\begin{pmatrix}
 K&C\\
 L_{\boldsymbol\gamma}&0
 \end{pmatrix}
 \prod_{\Delta\in T}^N z_{\Delta}^{f_\Delta''}(z_\Delta'')^{-f_\Delta}.
\end{equation}
where
\begin{enumerate}
    \item $C$ is the edge-vertex adjacency matrix of $\mathcal{T}$,
    \item $K=A\Delta_{z''}+B\Delta_{z}^{-1}$, $L_{\boldsymbol \gamma}=A_{\boldsymbol \gamma}\Delta_{z''}+B_{\boldsymbol \gamma}\Delta_{z}^{-1}$, where $A, B, A_{\boldsymbol \gamma}, B_{\boldsymbol \gamma}$ are Neumann-Zagier matrices and $\Delta_{z}$(resp. $\Delta_{z''}$) is the diagonal matrix whose diagonal entry are $\{z_{\Delta}|\Delta\in T\}$ (resp. $\{z_{\Delta}''|\Delta\in T\}$);
    \item $f_{\Delta}, f_{\Delta}'$ and $f_{\Delta}''$ form a combinatorial flattening if $\rho$ induces the complete hyperbolic structure; otherwise, they are required to form a strong flattening (See Section \ref{sec:2.3} for the definition of Neumann-Zagier datum).
\end{enumerate}

\end{theorem}
\begin{remark}
The identity \eqref{eq:main1} is different from the original formulation of the one-loop conjecture, in which the Jacobian term is defined using row-eliminated matrices, and here we use bordered matrices. The two formulations are easily proved to be equivalent. We adopt the current form not only because it does not involve any choice of edges to be eliminated, but also because it fits better with our proof scheme.
\end{remark}

The proof of Theorem \ref{thm:main1} is a direct comparison of based chain complexes. We begin with the relative adjoint twisted chain complex
$$\mathcal{A}_{\bullet}:=C_{\bullet}(M, \partial M;\mathrm{Ad}\rho).$$
We work with this chain complex for two reasons. (1)The ideal triangulation $\mathcal{T}$ defines a cell decomposition of the pair. (2) $\mathcal{A}_{\bullet}$ is the twisted Poincar\'e dual of the complex used to compute the torsion of $M$. Choosing a basis of homology groups of $\mathcal{A}_{\bullet}$ that is dual to Porti's basis, the torsion of $\mathcal{A}_{\bullet}$ equals the adjoint twisted Reidemeister torsion by a formula by Milnor (See Proposition \ref{thm:poincaredual}).

An infinitesimal M\"obius transform can be evaluated at every developed ideal vertex. Applied to the corners of the truncated cells, this gives an evaluation map from $\mathcal{A}_{\bullet}$ to a boundary complex $\mathcal{B}_{\bullet}$ whose coefficients are the tangent vectors at the developed vertices. The evaluation map suggests a more natural basis, which we call the geometric basis for the mapping cone 
$$\mathcal{K}_{\bullet}=\mathrm{Cone}(\mathcal{A}_{\bullet}\xrightarrow{\mathrm{ev}}\mathcal{B}_{\bullet-1}),$$
Then we show that
\begin{enumerate}
\item The determinant of the change from the cellular basis to the geometric basis provides the monomial flattening term in \eqref{eq:main1},
\item The mapping cone retracts to a two-term chain complex
$$\mathbb{C}^{T}\to \mathbb C^{E}\to 0$$
that capture the infinitesimal gluing equations of $(M, \mathcal{T})$. Its torsion provides the Jacobian term in \eqref{eq:main1}.
\end{enumerate}
\subsection{Turaev-Viro type invariants and their one-loop term}
The Turaev-Viro type invariants, e.g.,\cite{TV92,LMSWY25a,LMSWY26}, behave differently from those of Jones-Kashaev-Reshetikhin-Turaev type. They have the following features:
\begin{enumerate}
    \item They take values in real numbers instead of complex numbers.
    \item They are computed using triangulations of $3$-manifolds and $6j$-symbols of quantum groups.
    \item They are related to squared norms of Reshetikhin-Turaev invariants.
\end{enumerate}
Chen and Murakami \cite{CM23} first discovered that the Gram matrices of hyperbolic tetrahedra appeared in the constant term of the asymptotic expansion of quantum $6j$-symbols. The same phenomenon is observed for $b$-6j symbols \cite{LMSWY25a,LMSWY25b,LMSWY26}, i.e., the analogs of $6j$-symbols for representation theory of $U_{q\tilde{q}}(\mathfrak{sl}_{2})$. Based on the relation between the Turaev-Viro invariant and $6j$-symbols, Wong and Yang \cite{WY24} obtained geometric formulas for adjoint torsion involving Gram determinants and Jacobians of cone angles with respect to edge lengths. The formula was later applied to the study of asymptotics of the Turaev-Viro invariant constructed from $b$-6j symbols, which is a state-integral quantum invariant defined on hyperbolic manifolds with totally geodesic boundary.

In the study of $b$-6j symbols, the authors, together with Liu, Sun and Yang realize that $b$-6j symbols admit an interesting renormalization, so that the Gram matrices of decorated ideal tetrahedra appear in the constant term of the asymptotic expansion of the renormalized $6j$-symbols. Using these renormalized b-$6j$ symbols, the authors, together with Liu, Sun and Yang, extended the previously defined Turaev-Viro invariants for hyperbolic $ 3$-manifolds with cusped boundary, so it is natural to relate the constant term of the asymptotic expansion of the state-integral to the adjoint Reidemeister torsion. Our second theorem confirms the relation.

\begin{theorem}\label{thm:main2}
Let $M$ be a hyperbolic $3$-manifold with toroidal boundary, and let $\mathcal{T}$ be a geometric triangulation of $M$. Let $\boldsymbol \theta(\boldsymbol{x})=(\theta_{1},\ldots,\theta_{|E|})$ denote the cone angles around edges of $\mathcal{T}$ and let $Gram(\Delta)$ denote the Gram matrix of $\Delta\in T$, both considered as functions of decorated edge lengths $\boldsymbol{x}$. Then at $\boldsymbol{x}=\boldsymbol{x}_{0}$, the decorated edge length realizes the complete hyperbolic metric of $M$, we have
\begin{equation}\label{eq:main2}
\frac{1}{{\det(C^{T}C)^2}}\det\begin{bmatrix}
[\frac{d\theta_{i}}{dx_{j}}]_{ij} & C\\
C^{T} & 0
\end{bmatrix} e^{-\sum_{e} 2x_{e}}\prod_{\Delta\in T}4\sqrt{-\mathrm{Gram}(\Delta)}=\det[\mathrm{Im}(\frac{\partial H(l_{i})}{\partial H(m_{j})})]|\mathrm{tor}(M, \rho_{0},  \boldsymbol m)^{2}|,\end{equation}
where $\boldsymbol{m}=\{m_{1},\ldots, m_{|V|}\}$ and $\boldsymbol{l}=\{l_{1},\ldots, l_{V}\}$ are two systems of peripheral curves such that $m_{i}, l_{i}$ belong to the $i$-th vertex and their algebraic intersection number equals  $1$. Here $H(m_{1}),\ldots, H(m_{|V|})$ (resp. $H(l_{1}),\ldots, H(l_{|V|})$) denote the log-holonomies of curves $m_{1},\ldots, m_{|V|}$ (resp. $l_{1},\ldots, l_{|V|}$), considered as local coordinates of character varieties $\mathrm X(M)$ near the character $\rho_{0}$.
\end{theorem}

In \cite{WY24}, Wong and Yang anticipated a corresponding description for geometrically triangulated cusped manifolds in terms of decorated edge lengths. Theorem 1.3 provides a precise cusped counterpart, accounting for the decoration directions and the peripheral period matrix. The right-hand side also appears in the asymptotic expansion of the meromorphic 3D-index defined by Hodgson, Kricker and Siejakowski \cite{HKS21}, after we apply Theorem~\ref{thm:main1}. The explicit expression first appeared in \cite{HKS21}. We conjecture that the same term also appears in the asymptotic expansion of the Turaev-Viro invariant~\cite{TV92}.

The proof strategy of Theorem \ref{thm:main2} is to first rewrite the torsion in terms of shape parameters using \eqref{eq:main1}. We then interpret the squared norm of the complex Jacobian in \eqref{eq:main1} in terms of the Jacobian of real variables. This provides the Jacobian term of \eqref{eq:main2}, and the flattening monomial provides the Gram matrix.

The paper is organized as follows. Section 2 fixes the conventions for adjoint torsion, peripheral homology bases, and the torsion identities used in the proof, and recalls Neumann–Zagier data and decorated ideal tetrahedra. Section 3 constructs the evaluation comparison, computes its basis changes and residual differential, identifies the peripheral homology normalization, and proves Theorem 1.1. Section 4 derives Theorem 1.3 by comparing complex and real Jacobians and rewriting the resulting local factors in terms of Gram matrices.
\bigskip

\noindent\textbf{Acknowledgments.} We thank Tian Yang for helpful discussions and are especially grateful to him for conjecturing formula~\eqref{eq:main2} during one of these discussions.   S.M.\ was supported by the National Natural Science Foundation of China (No.12371124). B.W.\ was  supported by the National Key R$\&$D Program of China (No. 2023YFA1010700) and National Natural
Science Foundation of China Grant (No. 12526204).

\section{Preliminaries}
In this section, we recall the basics of adjoint twisted Reidemeister torsion for hyperbolic $3$-manifolds, the Neumann-Zagier datum, and decorated ideal hyperbolic tetrahedra. In Subsection \ref{sec:2.1} we recall the definition of adjoint Reidemeister torsion for both $3$-manifold $M$ and a pair $(M, \partial M)$, as well as Porti's choice of the bases of homology groups. In Subsection \ref{sec:2.2} we introduce three formulae for torsion of chain complexes, namely:
\begin{enumerate}
    \item Milnor's multiplicativity theorem (Theorem \ref{thm:milnor});
    \item the torsion formula for twisted Poincar\'e dual complex (Theorem \ref{thm:poincaredual});
    \item the torsion formula for an analog of Gaussian elimination (Lemma \ref{lem:retraction}).
\end{enumerate}
They are the main tools for our computation. In Subsection \ref{sec:2.3} we review the notation for the Neumann-Zagier datum for ideal triangulations of hyperbolic $3$-manifolds. In Subsection \ref{sec:2.4} and \ref{sec:2.5} we review angle structures, decorated ideal hyperbolic tetrahedra and their relation to the Neumann-Zagier datum.
\subsection{Adjoint Reidemeister Torsion}\label{sec:2.1}
\subsubsection{Torsion of chain complexes}
Let $C_{\bullet}$ be a based finite chain complex
$$0\to C_{d}\xrightarrow{\partial}C_{d-1}\xrightarrow{\partial}\ldots \xrightarrow{\partial}C_{1}\xrightarrow{\partial}C_{0}$$
of $\mathbb{C}$-vector spaces, and with a basis $\boldsymbol{c}_{k}$ for each $C_{k}$. Let $H_{\bullet}$ be the homology of $C_{\bullet}$, i.e.,
$$H_{i}:=\frac{\mathrm{Ker}(C_{i}\xrightarrow{\partial} C_{i-1})}{\mathrm{Im}(C_{i+1}\xrightarrow{\partial}C_{i})}.$$
For each $H_{k}$, choose a basis $\boldsymbol{h}_{\bullet}=\{\boldsymbol h_{0},\ldots, \boldsymbol{h}_{d}\}$ and a lift of $\widetilde{\boldsymbol{h}_{k}}\subset C_{k}$ of $\boldsymbol{h}_{k}$. We also choose basis $\boldsymbol{b}_{k}$ for each $\mathrm{Im}(C_{k+1}\to C_{k})$ and a lift $\widetilde{\boldsymbol{b}_{k}}\subset C_{k+1}$. Then the Reidemeister torsion of the chain complex $C_{\bullet}$ with respect to $\boldsymbol{h}_{\bullet}$ is defined by
$$\mathrm{tor}(C_{\bullet}, \{\boldsymbol{h}_{k}\})=\pm\prod_{k=0}^{d}[\boldsymbol{b}_{k}\sqcup \widetilde{\boldsymbol{b}_{k-1}}\sqcup\widetilde{\boldsymbol{h}_{k}};\boldsymbol c_{k}]^{(-1)^{k+1}}.$$
The torsion depends only on the choice of $\boldsymbol{h}_{k}$, and does not depend on the choice of $\boldsymbol{b}_{k}$ and the liftings $\widetilde{\boldsymbol{b}_{k}}$ and $\boldsymbol{h}_{k}$.
Now we recall the adjoint twisted Reidemeister torsion of a CW-complex. Let $K$ be a finite CW-complex of $M$ and let $\rho: \pi_{1}(M)\to PSL(2, \mathbb{C})$ be a representation of its fundamental group. Consider the twisted chain complex
$$C_{\bullet}(K; \mathrm{Ad}\rho)=C_{\bullet}(\widetilde{K}, \mathbb Z)\otimes_{\pi_{1}(M)}\mathfrak{sl}_{2}(\mathbb{C})$$
where $C_{\bullet}(\widetilde{K}, \mathbb Z)$ is the simplicial complex of the universal covering of $K$. The fundamental group acts on $C_{\bullet}(\widetilde{K}, \mathbb Z)$ by deck transformations and acts on $\mathfrak{sl}_{2}(\mathbb C)\cong \mathrm{Lie}(PSL(2, \mathbb{C}))$ by conjugation. Now we make it a based chain complex as follows. Let $\boldsymbol{e}_{1}, \boldsymbol{e}_{2}, \boldsymbol{e}_{3}$ be a basis of $\mathfrak{sl}_{2}(\mathbb{C})$ and $\{c_{1}^{k},\ldots, c_{i_{d}}^{k}\}$ be a lift of the set of $k$-cells of $C_{k}(K, \mathbb Z)$ to $C_{\bullet}(\widetilde{K}, \mathbb Z)$.
Then we call
$$\boldsymbol{c}_{k}=\{c_{s}^{k}\otimes \boldsymbol{e}_{i}|i\in\{1,2,3\}, s\in \{1,\ldots, i_{k}\}\}$$
a standard basis of $C_{k}(K, \rho)$. Choose a basis $\boldsymbol{h}_{k}$ for each $H_{k}(K, \rho)$. The adjoint twisted Reidemeister torsion of $M$ with respect to the basis $\boldsymbol{h}_{\bullet}$ is defined to be
$$\mathrm{tor}(M, \boldsymbol{h}_{\bullet};\mathrm{Ad}\rho)=\mathrm{tor}(C_{\bullet}(K; \mathrm{Ad}\rho), \boldsymbol{h}_{\bullet}).$$
By \cite{Por18}, $\mathrm{tor}(M, \boldsymbol{h}_{\bullet}; \mathrm{Ad}\rho)$ depends only on the conjugacy class of $\rho$ and the choice of $\boldsymbol{h}_{\bullet}$. It depends neither on the choice of basis $\boldsymbol{e}_{1}, \boldsymbol{e}_{2}, \boldsymbol{e}_{3}$ nor the choice of lifts of the cells. In addition, the torsion is invariant under elementary expansions, elementary collapses of CW-complexes and subdivisions. Hence it defines an invariant of PL-manifolds and of topological manifolds of dimension less or equal to $3$. Therefore, we will denote the adjoint twisted Reidemeister torsion of a $3$-manifold $M$ by $\mathrm{tor}(M, \boldsymbol{h}_{k};\mathrm{Ad}\rho)$.

The adjoint twisted Reidemeister torsion can also be defined for CW-pairs. We mainly focus on the case of CW-decomposition of the pair $(M, \partial M)$ given by an ideal triangulation $\mathcal{T}$ of a $3$-manifold $M$ with $\partial M$ homeomorphic to union of tori. Let $K$ be a CW-complex of pair $(M, \partial M)$. Consider the twisted chain complex
$$C_{\bullet}(K;\mathrm{Ad}\rho)=C(\widetilde{K}, \mathbb{Z})\otimes_{\pi_{1}(M)}\mathfrak{sl}_{2}(\mathbb C).$$
Here $\widetilde{K}$ is the lifted CW-complex of pair $(\widetilde{M}, \widetilde{\partial M})$, we abuse the notation by denoting the boundary of $\widetilde{M}$ by $\widetilde{\partial M}$. Similarly, let $\boldsymbol{h}_{\bullet}$ be a basis of the homology groups of $C_{\bullet}(K;\mathrm{Ad}\rho)$. We define the adjoint twisted Reidemeister torsion of a $3$-manifold relative to its boundary by 
$$\mathrm{tor}((M, \partial M), \boldsymbol{h}_\bullet;\mathrm{Ad}\rho):=\mathrm{tor}(C_{\bullet}(K;\mathrm{Ad}\rho),\boldsymbol{h}_{\bullet})$$

\subsubsection{Torsion as functions over character variety}
In this section we recall results of Porti \cite{Por18,Por97} for the Reidemeister torsions of hyperbolic $3$-manifolds twisted by the adjoint action $\mathrm {Ad}_\rho=\mathrm {Ad}\circ\rho$ of an irreducible $\mathrm {PSL}(2;\mathbb C)$-representation $\rho$. Here $\mathrm {Ad}$ is the adjoint action of $\mathrm {PSL}(2;\mathbb C)$ on its Lie algebra $ \mathfrak{sl}_2(\mathbb C).$

Now suppose  $M$ is a compact, orientable  $3$-manifold with boundary consisting of $n$ disjoint tori $T_1, \dots,  T_n$ whose interior admits a complete hyperbolic structure with  finite volume. Let $\mathrm X(M), \mathrm X_{0}(M)$, $\mathrm X^{n}(M)$ and $\mathrm X^{irr}(M)$ be the $\mathrm{PSL}(2;\mathbb C)$-character variety of $M$, the distinguished component containing $\rho_{0}$, the components of dimension $n$, and the subset containing irreducible characters.

\begin{theorem}\cite[Section 3.3.3]{Por97}\label{HM} For a system of simple closed curves $\boldsymbol\gamma=(\gamma_1,\dots,\gamma_n)$ on $\partial M$ with $\gamma_i\subset T_i,$ $i\in\{1,\dots,n\},$  and a character $[\rho]$ in a Zariski open subset of $\mathrm X_0(M)\cap\mathrm X^{\text{irr}}(M),$  we have:
\begin{enumerate}[(i)]
\item For $k\neq 1,2,$ $ H_k(M;\mathrm{Ad}\rho)=0.$
\item  For $i\in\{1,\dots,n\},$ up to scalar $\mathrm Ad_\rho(\pi_1(T_i))^T$ has a unique invariant vector $\mathbf I_i\in \mathbb C^3;$ and
$$\mathrm H_1(M;\mathrm{Ad}\rho)\cong \mathbb C^n$$ 
with a basis
$$ h^1_{(M,\gamma)}=\{\mathbf I_1\otimes [\gamma_1],\dots, \mathbf I_n\otimes [\gamma_n]\}$$
where $([\gamma_1],\dots,[\gamma_n])\in  H_1(\partial M;\mathbb Z)\cong 
\bigoplus_{i=1}^n H_1(T_i;\mathbb Z).$
 
\item Let $([T_1],\dots,[T_n])\in \bigoplus_{i=1}^n H_2(T_i;\mathbb Z)$ be the fundamental classes of $T_1,\dots, T_n.$ Then 
 $$ H_2(M;\mathrm{Ad}\rho)\cong\bigoplus_{i=1}^n H_2(T_i;\mathrm{Ad}\rho)\cong \mathbb C^n$$ 
with  a basis 
$$ h^2_M=\{\mathbf I_1\otimes [T_1],\dots, \mathbf I_n\otimes [T_n]\}.$$
\end{enumerate}
\end{theorem}

\begin{definition}\label{reg} Let $\boldsymbol\gamma=(\gamma_1,\dots,\gamma_n)$ be a system of simple closed curves on $\partial M$ with $\gamma_i\subset T_i,$ $i\in\{1,\dots,n\}.$  A character $[\rho]$ in $\mathrm X^n(M)\cap\mathrm X^{\text{irr}}(M)$ is \emph{$\boldsymbol\gamma$-regular} if condition (ii) in Theorem \ref{HM} is satisfied.
\end{definition}

It follows that for any system of simple closed curves $\boldsymbol \gamma$ on $\partial M,$  the $\boldsymbol \gamma$-regular characters are smooth points of $\mathrm X(M);$ and the logarithmic holonomies of $\boldsymbol \gamma$  form local coordinates of $\mathrm X(M)$ near each of the $\boldsymbol \gamma$-regular characters. Here for a $\mathrm{PSL}(2;\mathbb C)$-character $[\rho],$ the logarithmic holonomy $H(\gamma_{i})$ of $\gamma_i$ is defined up to sign as the logarithm of the ratio of the eigenvalues of $\rho([\gamma_i]).$

\begin{definition} \label{ATRT}
The adjoint twisted Reidemeister torsion of $M$ with respect to $\boldsymbol\gamma$ is the function 
$$\mathrm{tor}(M, -, \boldsymbol{\gamma}): \mathrm X^n(M)\cap\mathrm X^{\text{irr}}(M)\to\mathbb C/\{\pm 1\}$$ 
defined by
$$\mathrm{tor}(M, \rho, \boldsymbol{\gamma})=\mathrm{tor}(M, \{\mathbf h^1_{(M,\gamma)},\mathbf h^2_M\};\mathrm{Ad}_\rho)$$
if $\rho$ is $\boldsymbol \gamma$-regular, and by $0$ otherwise. 
\end{definition}

\begin{theorem}\cite[Theorem 4.1, see also Theorem 2.8 of \cite{WY24}]{Por97}\label{thm:rationalfunction}
Let $M$ be a compact, orientable 3-manifold with boundary consisting of $n$ disjoint tori whose interior admits a complete hyperbolic structure with finite volume. Let $\boldsymbol{\gamma}$ be a system of simple closed curves on $\partial M$, Then $\mathrm{tor}(M, -, \boldsymbol{\gamma})$ is a rational function over $\mathrm X^{n}(M)\cap \mathrm  X^{irr}(M)$. In addition, if a component $\mathrm X_{k}(M)$ contains one $\boldsymbol{\gamma}$-regular character, then the support of $T_{(M,\boldsymbol\gamma)}$ constains a Zariski-open subset of $\mathrm X_{k}(M)\cap \mathrm X^{irr}(M)$. 
\end{theorem}

\subsection{Torsion Formulas}\label{sec:2.2}
The main computational tool for torsion is Milnor's multiplicativity theorem \cite{Mil66}; see also \cite{Por97}.

\begin{theorem}\label{thm:milnor}[Milnor \cite{Mil66}]
Let
$$0\to C_{\bullet}'\to C_{\bullet}\to C_{\bullet}''\to 0$$
be a short exact sequence of finite based chain complexes. Fix homology bases $\boldsymbol{h}_{\bullet}, \boldsymbol{h}_{\bullet}'$ and $\boldsymbol{h}_{\bullet}''$ for $C_{\bullet}$, $C_{\bullet}'$ and $C_{\bullet}''$ respectively. Let $\mathcal{H}_{\bullet}$ be the associated long exact sequence of homology groups, viewed as a based acyclic chain complex with basis $\{\boldsymbol{h}_{\bullet}', \boldsymbol{h}_{\bullet}, \boldsymbol{h}_{\bullet}''\}$. Then
$$\mathrm{tor}(C_{\bullet}; \boldsymbol{h}_{\bullet})=\pm\mathrm{tor}(C_{\bullet}'; \boldsymbol{h}_{\bullet}')\mathrm{tor}(C_{\bullet}''; \boldsymbol{h}_{\bullet}'')\mathrm{tor}(\mathcal{H}_{\bullet}).$$
\end{theorem}

Another torsion formula is the twisted Poincar\'e duality, which identifies the Adjoint twisted Reidemeister torsion of $M$ and that of $(M, \partial M)$.

Suppose $K$ is a cell-decomposition of a $3$-manifold $M$ with nonempty boundary. Then its dual complex $K^{\dagger}$ defines a cell decomposition of the pair $(M, \partial M)$. The cellular intersection together with the following invariant bilinear form
$$B(X, Y):=\mathrm{tr}(XY)$$
of $\mathfrak{sl}_{2}(\mathbb C)$ induces perfect pairings
$$C_{3-i}(K, \rho)\times C_{i}(K^{\dagger}, \rho)\to \mathbb{C}.$$
By the twisted Poincar\'e duality. This induces a perfect pairing of homology
\begin{equation}\label{eqn:poincare}
H_{i}(M;\mathrm{Ad}\rho)\otimes H_{3-i}(M, \partial M;\mathrm{Ad}\rho)\to \mathbb{C}.
\end{equation}
The following theorem by Milnor \cite{Mil62} relates the torsion of the dual chain complexes.
\begin{theorem}\label{thm:poincaredual}
Let $\boldsymbol{h}_{k}$ be basis of $H_{i}(M;\mathrm{Ad}\rho)$, and let $\boldsymbol{h}_{3-k}^{\dagger}$ be dual basis with respect to the perfect pairing \eqref{eqn:poincare}. Then
\begin{equation}
\mathrm{tor}(M, \boldsymbol{h}_{k};\mathrm{Ad}\rho)=\pm \mathrm{tor}(M, \partial M, \boldsymbol{h}_{k}^{\dagger};\mathrm{Ad}\rho)
\end{equation}
\end{theorem}

Doing Gaussian elimination by matrices of determinant $1$ does not change the determinant. The next torsion formula is the chain-complex analog.
\begin{lemma}\label{lem:retraction}
Let $C_{\bullet}$ be a chain complex, and for some $i$, $C_{i}=C_{i}'\oplus D_{i}$ and $C_{i-1}=C_{i-1}'\oplus D_{i-1}$. Suppose 
$$\partial_{i}^{C}=\begin{bmatrix}
\alpha & \beta\\
\gamma & \delta
\end{bmatrix},$$
and $\delta: D_{i}\to D_{i-1}$ is an isomorphism. Then ${C}_{\bullet}$ is chain-homotopy equivalent to
$$C'_{\bullet}:=(\ldots\to C_{i-1}\xrightarrow{p_{C_{i+1}'}\circ \partial_{i+1}^{C}} C_{i}'\xrightarrow{\alpha-\beta\circ \delta^{-1}\circ \gamma}C_{i-1}\xrightarrow{\partial_{i-1}^{C}|_{C_{i}'}} C_{i-2}\to\ldots),$$
where $p_{C_{i+1}'}$ is the projection of $C_{i}$ to $C_{i}'$. In addition, if $C_{\bullet}$ is a based chain complex and $\delta=I$ with respect to the chosen bases. Then
$$\mathrm{tor}(C_{\bullet}, \boldsymbol{h}_{\bullet})=\pm\mathrm{tor}(C_{\bullet}', \boldsymbol{h}_{\bullet}).$$
Here we identify the homology groups of $C_{\bullet}$ and $C_{\bullet}'$ since they are homotopy equivalent.
\end{lemma}
The first assertion is the Gaussian elimination lemma for chain complexes; See \cite[Lemma 3.2]{Bar07}. The torsion assertion follows from the standard functoriality of the torsion of based complexes; see \cite{Tur01}.

\subsection{Neumann-Zagier Datum and gluing structure}\label{sec:2.3}
Let $M$ be the interior of a compact oriented $3$-manifold whose boundary is homeomorphic to union of $k$ tori, and let $\mathcal{T}$ be an ideal triangulation with its set of vertices, edges, faces and tetrahedron denoted by $V, E, F$ and $T$ respectively. Choose a quad type in each tetrahedron and let $G, G'$ and $G''$ be the edge-quad adjacency matrices, they are of size $|E|\times |T|$. Let $\gamma$ be a boundary peripheral curve. Define $|T|$-dimensional row vectors $G_{\gamma}, G_{\gamma}'$ and $G_{\gamma}''$ recording the corners on the left of $\gamma$ minus the corners on the right of $\gamma$. Denote the matrices
$$\begin{array}{cc}
    A=G-G', & B=G''-G'. \\
    A_{\gamma}=G_{\gamma}-G_{\gamma}', & B_{\gamma}=G_{\gamma}''-G_{\gamma}'.
\end{array}$$
These matrices have the following properties that come from the symplectic nature of Neumann-Zagier datum.
\begin{proposition}[\cite{DG13}]\label{prop:NZsymplectic}
$$AB^{T}-BA^{T}=0,\quad A_{\gamma}B^{T}-B_{\gamma}A^{T}=0$$
and for any pair of oriented peripheral curves $m$ and $l$, one have
$$A_{m}B_{l}^{T}-B_{m}A_{l}^{T}=2i(m, l),$$
where $i(m, l)$ is the algebraic intersection of $m$ and $l$ on $\partial M$.
\end{proposition}
The Neumann-Zagier datum is useful for studying the deformation space of $M$. Associate to each tetrahedron $\Delta$ a shape parameter $z_{\Delta}\in \mathbb{C}\setminus\{0, 1\}$. Also define
$$z_{\Delta}'=\frac{1}{1-z_{\Delta}},\quad z_{\Delta}''=1-\frac{1}{z_{\Delta}}.$$
They are the shape parameters corresponding to the other two quad types of $\Delta$. The deformation space $\mathcal{V}(\mathcal{T})$ is defined to be the subset of the space of shapes $\mathcal Z(\mathcal{T})$ satisfying the gluing equations. In terms of Neumann-Zagier datum, the deformation space
$$\mathcal{V}(\mathcal{T})=\{z_{\Delta}\in \mathbb{C}\setminus\{0, 1\}|\sum_{\Delta} G_{e\Delta}\log z_{\Delta}+G_{e\Delta}'\log z_{\Delta}'+G_{e\Delta}''\log z_{\Delta}''=2\pi i \text{ for all $e$}\}.$$
If the tetrahedra with shapes $z_{\Delta}$ glue to a hyperbolic $3$-manifold with complete hyperbolic structure, then the completeness equation is also satisfied. In terms of Neumann-Zagier datum, for every peripheral curve $\gamma\subset \partial M$, we have
\begin{equation}\label{eq:completeness}
G_{\gamma\Delta}\log z_{\Delta}+G_{\gamma\Delta}'\log z_{\Delta}'+G_{\gamma\Delta}''\log z_{\Delta}''=0.
\end{equation}
A combinatorial flattening of $\mathcal{T}$ is an assignment of integers $f_{\Delta}, f_{\Delta}'$ and $f_{\Delta}''$ to the quads of each tetrahedron $\Delta$, subject to the following conditions:
$$f_{\Delta}+f_{\Delta}'+f_{\Delta}''=1,\quad \sum_{\Delta} G_{e\Delta}f_{\Delta}+G_{e\Delta}'f_{\Delta}'+G_{e\Delta}''f_{\Delta}''=2.$$
A flattening is called strong if
\begin{align*}\label{eq:strongflat}
\sum_{\Delta} G_{\gamma\Delta}f_{\Delta}+G_{\gamma\Delta}'f_{\Delta}'+G_{\gamma\Delta}''f_{\Delta}''=0
\end{align*}
for any peripheral curve $\gamma\subset \partial{M}$. It is proved in \cite[Lemma 6.1]{Neu92} that a strong flattening exists for any ideal triangulation.
\subsection{Angle structures}\label{sec:2.4}
Let $(M, \mathcal{T})$ be an ideally triangulated $3$-manifold with its boundary homeomorphic to a union of tori. An angle assignment on $(M, \mathcal{T})$ assigns each pair $(\Delta, e)$ such that $\Delta$ is adjacent to $e$ a number $\theta_{(\Delta, e)}\in (0, \pi)$, called the dihedral angle of $e$ in $\Delta$, so that the sum of the dihedral angles at three edges in the same tetrahedron $\Delta$ adjacent to each vertex equals $\pi$. These are exactly the conditions for six numbers in $(0, \pi)$ to be the dihedral angles of an ideal hyperbolic tetrahedron. The cone angle of an angle assignment is the assignment $(\theta_{1}, \ldots, \theta_{|E|})\in \mathbb R^{E}$ that assigns to each edge $e$ the sum of the dihedral angles at $e$. They are linear functionals on the space of angle assignments. Since the dihedral angles of an ideal tetrahedron are uniquely determined by its shape parameter $z\in \{\mathrm{Im} z>0\}$, there is a natural map from the space of shape parameters $\mathcal{Z(T)}$ to the space of angle assignments. The cone angles can be computed using Neumann-Zagier data:

\begin{equation}
\theta_{e}=\mathrm{Im}(\sum_{\Delta\in T}G_{e\Delta}\log z+G_{e\Delta}'\log z'+G_{e\Delta}''\log z'' )
\end{equation}
\subsection{Decorated ideal hyperbolic tetrahedra and Gram matrices}\label{sec:2.5}
A decoration of an ideal hyperbolic tetrahedron is an assignment of a horosphere for each ideal vertex. We call an ideal hyperbolic tetrahedron together with a decoration a decorated ideal hyperbolic tetrahedron. For a decorated ideal hyperbolic tetrahedron, the edge length $x_{ij}$ of the edge $e_{ij}$ is defined to be the signed distance between the points of intersection of $e_{ij}$ and the horospheres at the $i$-th and $j$-th vertices.

For a decorated ideal hyperbolic tetrahedron $\Delta$ with edge lengths $\boldsymbol{x}=(x_{12},\ldots, x_{34})$, the Gram matrix in the edge lengths is defined by
\begin{equation*}
\mathrm{Gram}(\Delta)=\left[\begin{array}{cccc}
 0&-\frac{e^{x_{12}}}{2}&-\frac{e^{x_{13}}}{2}& -\frac{e^{x_{14}}}{2}\\
-\frac{e^{x_{21}}}{2} &0&-\frac{e^{x_{23}}}{2}& -\frac{e^{x_{24}}}{2}\\
 -\frac{e^{x_{31}}}{2}&-\frac{e^{x_{32}}}{2}&0& -\frac{e^{x_{34}}}{2} \\
 -\frac{e^{x_{41}}}{2}&-\frac{e^{x_{42}}}{2}&-\frac{e^{x_{43}}}{2}&0
\end{array}\right]
\end{equation*}
A decorated ideal hyperbolic tetrahedron is determined by the edge lengths, as are the shape parameters $z_{\Delta}, z_{\Delta}'$ and $z_{\Delta}''$ of the underlying ideal hyperbolic tetrahedron. In particular, the absolute values of shape parameters could be expressed  simply in terms of edge lengths:
\begin{equation}\label{eqn:zinlength}|z_{\Delta}|=e^{\frac{x_{02}+x_{13}-x_{03}-x_{12}}{2}},\quad |z_{\Delta}'|=e^{\frac{x_{03}+x_{12}-x_{01}-x_{23}}{2}},\quad |z_{\Delta}''|=e^{\frac{x_{01}+x_{23}-x_{02}-x_{13}}{2}}.
\end{equation}
\begin{proposition}\label{prop:grambyz}
$$\det(\mathrm{Gram}(\Delta))=-\frac{1}{4}e^{2x_{03}+2x_{12}}\mathrm{Im}(z_{\Delta})^{2}$$
\end{proposition}
\begin{proof}
By direct computation,
\begin{align*}
\det(\mathrm{Gram}(\Delta))=\frac{1}{16}(e^{2b}+e^{2b'}+e^{2b''}-2e^{b+b'}-2e^{b+b''}-2e^{b'+b''})
\end{align*}
where $b=x_{01}+x_{23}, b'=x_{02}+x_{13}$ and $b''=x_{03}+x_{12}$. Use \eqref{eqn:zinlength}, it equals to
\begin{align*}
&\frac{1}{16}e^{2b''}(|z'|^{-4}+|z|^{4}+1-2|z|^{2}|z'|^{-2}-2|z|^{2}-|z'|^{-2})\\
=&\frac{1}{16}e^{2b''}(|1-z|^{4}+|z|^{4}+1-2|z|^{2}|1-z|^{2}-2|z|^{2}-2|1-z|^{2})\\
=&\frac{1}{16}e^{2b''}\Big((\mathrm{Re}(z)-1)^{2}+\mathrm{Im}(z)^{2})^{2}+((\mathrm{Re}(z))^{2}+\mathrm{Im}(z)^{2})^{2}+1-\\
&\hspace{1.5cm}2((\mathrm{Re}(z)-1)^{2}+\mathrm{Im}(z)^{2})((\mathrm{Re}(z)-1)^{2}+\mathrm{Im}(z)^{2})-\\
&\hspace{1.5cm}2((\mathrm{Re}(z)-1)^{2}+\mathrm{Im}(z)^{2})-2((\mathrm{Re}(z)-1)^{2}+\mathrm{Im}(z)^{2}\Big)\\
=&-\frac{1}{4}e^{2b''}\mathrm{Im}(z)^2
\end{align*}
\end{proof}
\section{Proof of Theorem \ref{thm:main1}}
In this section, we prove the following theorem. Theorem \ref{thm:main1} then follows from Theorem \ref{thm:poincaredual} and the fact the characters satisfying the conditions of \ref{thm:mainincomplete} form a dense subset of the $\boldsymbol \gamma$-regular locus of $X(M)$.

\begin{theorem}\label{thm:mainincomplete}
Let $M$ be a hyperbolic $3$-manifold with its boundary $\partial M$ is homeomorphic to a union of tori $T_{1}\ldots T_{|V|}$ and let $\boldsymbol{\gamma}=\{\gamma_{i}|i=1,\ldots |V|\}$ be a set of peripheral curves with $\gamma_{i}\subset T_{i}$. Then for any $\boldsymbol{z}\in \mathcal{V}(\mathcal{T})$ such that the induced $PSL(2, \mathbb{C})$-representation is $\boldsymbol{\gamma}$-regular and the induced character 
$$\chi_{\rho}:\pi_{1}(\partial M)\to \mathbb C^{*}$$
is nontrivial for each of boundary torus $\partial_{i}M=T_{i}$, we have
$$\mathrm{tor}(M, \partial M, \boldsymbol{h}_{\bullet}^{\gamma};\mathrm{Ad}\rho)
 =
 \frac{1}{\det(C^TC)}
 \det\begin{pmatrix}
 K&C\\
 L_{\boldsymbol\gamma}&0
 \end{pmatrix}
 \prod_{j=1}^N z_j^{f_j''}(z_j'')^{-f_j}.
$$
where
\begin{enumerate}
    \item $\boldsymbol{h}_{\bullet}^{\gamma}=\{\boldsymbol h_{1}^{\gamma}, \boldsymbol h_{2}^{\gamma}\}\subset H_{\bullet}(M, \partial M;\mathrm{Ad}\rho)$ is chosen such that $\boldsymbol h_{1}^{\gamma}$ and $\boldsymbol h_{2}^{\gamma}$ are bases dual to $\{[T_{i}]\otimes I\}$ and $\{[ \gamma_{i}]\otimes I\}$ respectively,
    \item $K$ and $L_{\boldsymbol \gamma}$ are the same as those in \ref{thm:main1}.
    \item $f_{\Delta}, f_{\Delta}'$ and $f_{\Delta}''$ form a \textbf{strong} combinatorial flattening.
\end{enumerate}
 
\end{theorem}

We prove the theorem by introducing an evaluation map from the adjoint twisted CW complex of $(M, \partial M)$ with its cell decomposition given by $\mathcal{T}$ to a twisted CW chain complex of $\partial M$. We then apply Milnor's multiplicativity theorem to its mapping cone. The reason we choose to work with the characters that restrict to non-trivial characters on the boundary tori is simple: For such characters, the induced twisted homology groups of boundary tori is trivial, which makes the long exact sequence of homology groups in the Milnor's multiplicativity theorem simpler.

\subsection{The evaluation map}
In this subsection, we describe two based chain complexes $(\mathcal{A}_{\bullet}, d_{\mathcal{A}})$ and $(\mathcal{B}_{\bullet}, d_{\mathcal{B}})$ used for computing $\mathrm{tor}(M, \partial M;\mathrm{Ad}\rho)$ and the evaluation map between them.

The ideal triangulation $\mathcal{T}$ defines a CW chain complex $C_{\bullet}(M, \partial M)$ for the pair $(M, \partial M)$, we denote the twisted chain complex $C_{\bullet}(\widetilde{M}, \widetilde{\partial M})\otimes_{\rho}\mathfrak{sl}(2, \mathbb C)$ by $\mathcal{A}_{\bullet}$, and choose the following basis 
$$\{\Delta\otimes \mathfrak e, \Delta\otimes \mathfrak f, \Delta\otimes \mathfrak h|t_{i}\in \bar T\}\subset \mathcal{A}_{3}$$
$$\{f\otimes \mathfrak e, f\otimes \mathfrak f, f\otimes \mathfrak h|f\in \bar F\}\subset \mathcal{A}_{2}$$
$$\{e\otimes \mathfrak e, e\otimes \mathfrak f, e\otimes \mathfrak h|e_{i}\in \bar E\}\subset \mathcal{A}_{1}$$
for $\mathcal{A}_{\bullet}$. Here $\bar T, \bar F$ and $\bar E$ are arbitrary lifts of $3$-cells in $T$, $2$-cells in $F$ and $1$-cells in $E$ respectively in the universal cover, and
$$\mathfrak e=\begin{bmatrix}
    0& 1\\
    0 & 0
\end{bmatrix}, \quad \mathfrak f=\begin{bmatrix}
    0& 0\\
    1 & 0
\end{bmatrix}, \quad \mathfrak h=\begin{bmatrix}
    1& 0\\
    0 & -1
\end{bmatrix}$$
be the standard basis of $\mathfrak{sl}_{2}(\mathbb{C})$. With this choice of basis, by the definition of Reidemeister torsion of pair $(M, \partial M)$, we have
$$\mathrm{tor}(\mathcal{A}_{\bullet},\boldsymbol{h}_{\bullet}^{\gamma})=\mathrm{tor}(M, \partial M, \boldsymbol{h}_{\bullet}^{\gamma};\mathrm{Ad}\rho).$$

Restricting $\mathcal{T}$ to boundary tori defines a triangulation of $\partial M$, which lifts to a triangulation of $\widetilde{\partial M}$. Choose $\boldsymbol{z}\in \mathcal{V(T)}$ that is compatible with $\rho$, and a developing map $D_{\boldsymbol{z}}$ sending $(\widetilde{M}, \widetilde{\partial M})$ to $\mathbb H^{3}$. Then each cell $c$ of $\widetilde{\partial M}$ is sent to a cell contained in a horospheres centered at $v(c)\in \partial \mathbb{H}^{3}\cong \mathbb CP^{1}$. Now take
$$\mathcal{B}_{i}:=\mathrm{span}_{\mathbb{C}}\{c\otimes T_{v(c)}\mathbb CP^{1}|c\in C_{i}(\widetilde{\partial M})\}/\pi_{1}(M).$$
Here the action of the fundamental group $\pi_{1}(M)$ on the tangent bundle $T\mathbb CP^{1}$ is induced by the representation $\rho: \pi_{1}(M)\to PSL(2, \mathbb{C})$.

Identify $\mathfrak{sl}_{2}(\mathbb{C})$ with the Lie algebra of $PSL(2, \mathbb{C})$. Each element $X\in \mathfrak{sl}_{2}(\mathbb{C})$ induces a vector field on $\partial \mathbb H^{3}\cong \mathbb CP^{1}$ via the infinitesimal action. We define the following evaluation map
\begin{align*}
    ev_{n}: \mathcal{A}_{n} &\to \mathcal{B}_{n-1}\\
    c\otimes X &\mapsto \sum_{i=0}^{n}\partial_{i}c\otimes X|_{p(\partial_{i}c)}
\end{align*}
where $\partial_{i}c\in C_{n-1}(\widetilde{\partial M})$ denotes the $i$-th corner (i.e., the $(n-1)$-cell at $i$-th vertex) of $c$, and $X|_{v(\partial_{i}c)}$ denotes the value of the vector field $X$ at the $v(\partial_{i}c)\in \partial \mathbb H^{3}$.

Now we describe a basis for $\mathcal{B}_{\bullet}$. For any $\mathbb C^{2}$ lifting of a vertex $v$ of $\widetilde{M}$, there is a natural identification between the tangent space of a vertex $v$ and $\mathrm{hom}(\langle v \rangle , \mathbb C^{2}/\langle v \rangle).$
After choosing a lifting $\tilde{v}$ of $v$, We define the vector $\epsilon_{\tilde v}$ in $T_{v}\mathbb CP^{1}$ to be the linear map satisfying
$$\det(\epsilon_{\tilde v}(\tilde v), \tilde v)=1.$$
Since $\dim \mathbb CP^{1}=1$, the vector $\epsilon_{\tilde v}$ is completely determined by the lifting. If vertex $v$ is instead lifted to the vector $\tilde{v}'=a\tilde{v}$ for some scalar $a$, then $\epsilon_{\tilde{v}'}=a^{-2}\epsilon_{\tilde v}$. Now we define the basis of $\mathcal{B}_{\bullet}$. We first choose a nice celluar fundamental domain in the universal cover. The following proposition garantees the existence of the desired fundamental domain.
\begin{proposition}\label{prop:lifttorus}
The induced triangulation on $\partial_{i}M$ admits a cellular fundamental domain $D_{i}\subset \widetilde{\partial_{i}M}$, satisfying one of the following conditions.
\begin{enumerate}
\item The boundary $\partial D_{i}$ consists of four consecutive paths $m_{i}, l_{i}, m_{i}', l_{i}'$, with all endpoints projecting to the same point $\star_{i}$ in $\partial_{i}M$, and $m_{i}'$ (resp. $l_{i}'$) is identified with $m_{i}^{-1}$(resp. $l_{i}^{-1}$) in $\partial_{i}M$. See Figure \ref{fig:42}.
\item The boundary $\partial D_{i}$ consists of six consecutive paths $m_{i},a_{i}, l_{i}, m_{i}', a_{i}', l_{i}'$, and the six end points projecting to two distict points $\star_{i}$ and $\star_{i}'$ in $\partial_{i}M$, with $a_{i}, m_{i}$ and $l_{i}$ are identified with $a_{i}', m_{i}'$, $l_{i}'$ respectively with their orientation reversed. See Figure \ref{fig:33}.
\end{enumerate}
We call the first case type $(4, 2)$ and the second case type $(3, 3)$.
\end{proposition}

\begin{proof}
The statement is equivalent to the existence of a graph $\Gamma_{i}$ belongs to the $1$-skeleton of $\partial_{i}M$, such that
\begin{enumerate}
\item $\partial_{i}M\setminus \Gamma_{i}$ is homeomorphic to a disk.
\item $\Gamma_{i}$ is either a theta graph or a join of $2$-circles.
\end{enumerate}
Now we describe an algorithm for finding such a graph $\Gamma_{i}$:
\begin{enumerate}
    \item Find a simple closed curve $\Gamma_{i}^{0}$ in the $1$-skeleton of $\partial_{i}M$, it cut $\partial_{i}M$ to an annulus.
    \item Find a path $\Gamma_{i}^{1}$ connecting the two components of the annulus obtained in the last step.
    \item Then $\Gamma_{i}=\Gamma_{i}^{0}\cup \Gamma_{i}^{1}$ is the desired graph.
\end{enumerate}
We complete the proof.
\end{proof}
Now we describe the basis of $\mathcal{B}_{\bullet}$.
\begin{enumerate}
    \item For each boundary torus $\partial_{i}M$, choose a cellular fundamental domain $D_{i}\subset \widetilde{\partial M}$ satisfying the conditions described in \ref{prop:lifttorus}. We define a set $S$ containing cells in the closure of the fundamental domain $\cup_{i\in V}D_{i}$.
    \begin{enumerate}
        \item[Type (4, 2):] Let $S$ contain all cells in the interior of $D_{i}$ for all $i$, together with all cells in the interior of $m_{i}\cup l_{i}$ for all $i$.
        \item[Type (3, 3):]  Let $S$ contain all cells in the interior of $D_{i}$ for all $i$, together with all cells in the interior of $m_{i}\cup l_{i}\cup a_{i}$ for all $i$.
    \end{enumerate}
\begin{figure}
    \centering
\begin{subfigure}[b]{0.48\textwidth}
     \centering
$\vcenter{\hbox{
\begin{tikzpicture}
\draw[->-=0.18, ->-=0.42, ->-=0.68, ->-=0.92] (0, 0)node[below]{$\star_{i}$} to node[below]{$l_{i}$} (3, 0) to node[right]{$m_{i}'$} (3, 1.5) to node[above]{$l_{i}'$} (0, 1.5) to node[left]{$m_{i}$} (0, 0);
\end{tikzpicture}}}$
\caption{Type (4,2): The fundamental domain consists of $4$ consecutive edges. $m_{i}'$ could be obtained from $m_{i}$ by applying deck transformation along $l_{i}$ and reversing the orientation. $l_{i}'$ could be obtained from $l_{i}$ by applying a deck transformation along $m_{i}^{-1}$ and reversing the orientation.\vspace{8.4mm}}
\label{fig:42}
\end{subfigure}
\hfill\quad
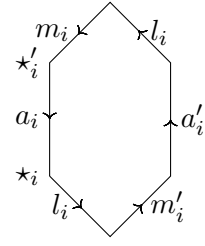
\begin{subfigure}[b]{0.48\textwidth}
        \centering
$\vcenter{\hbox{
\begin{tikzpicture}
\draw[->-=0.08, ->-=.23, ->-=.40, ->-=.58,->-=.73,->-=.90] (0, 0)node[left]{$\star_{i}$} to node[left]{$l_{i}$} (0.8, -0.8) to node[right]{$m_{i}'$} (1.6, 0) to node[right]{$a_{i}'$} (1.6, 1.5) to node[right]{$l_{i}$} (0.8, 2.3) to node[left]{$m_{i}$} (0, 1.5) node[left]{$\star_{i}'$} to node[left]{$a_{i}$} (0, 0);
\end{tikzpicture}}}$
\caption{Type (3,3): The fundamental domain consists of $6$ consecutive edges. $m_{i}'$ could be obtained from $m_{i}$ by applying deck transformation along $l_{i}\circ a_{i}$ and reversing the orientation. $l_{i}'$ could be obtained from $l_{i}$ by applying deck transformation along $a_{i}^{-1}\circ m_{i}^{-1}$ and reversing the orientation. $a_{i}'$ could be obtained from $a_{i}$ by applying deck transformation along $m_{i}^{-1}\circ l_{i}\circ a_{i}$ and reversing the orientation.}
\label{fig:33}
\end{subfigure}
\caption{}
\label{fig:1}
\end{figure}
    \item Under a developing map $\widetilde{M}\to \mathbb H^{3}$. the cellular fundamental domain $D_{i}$ belongs to a vertex of $\partial \mathbb{H}^{3}\cong \mathbb CP^{1}$. Lift it to $v_{i}\in \mathbb C^{2}$ for each $i$.
    \item Now for each corner $\partial_{k}c$ of a cell $c$ belonging to $\bar T, \bar F$ and $\bar E$, there is a unique element $g\in \pi_{1}(M)$, and a unique cell $d\in S$ such that 
    $$g(d)=\partial_{k}c,$$
    we lift the vertex corresponding to $\partial_{k}c$ to a vector $v(\partial_{k} c)$ satisfying
    \begin{equation}\label{eqn:transporting}
    v(\partial_{k} c)=\eta_{i} \rho(g)(v_{i}),
    \end{equation}
    where $\eta_{i}=\pm 1$; the sign ambiguity comes from $\rho(g)\in PSL(2, \mathbb{C})$.
    \item With the choices of vectors $v(\partial_{k}c)$ associated to each corner of cells of $\bar T, \bar F$ and $\bar E$, define bases 
    $$\{\partial_{i}\Delta\otimes \epsilon_{v(\partial_{i}\Delta)}|\Delta\in \bar T, i=0, 1, 2, 3\},$$
    $$\{\partial_{i}f\otimes \epsilon_{v(\partial_{i}f)}|f\in \bar F, i=0, 1, 2\},$$
    $$\{\partial_{i}e\otimes \epsilon_{v(\partial_{i}e)}|e\in \bar E, i=0, 1\},$$
    for $\mathcal{B}_{2}$, $\mathcal{B}_{1}$ and $\mathcal{B}_{0}$ respectively.
\end{enumerate}
With these choices of bases and lifts of vertices of the image of the developing map, we obtain the following corollary, which allows us to transfer corners of cells in $\bar T, \bar F$ and $\bar E$ to cells in $S$.
\begin{corollary}\label{cor:equivariantB}
Let $c$ be any cell of $\bar T\cup \bar F\cup \bar E$, and suppose $d$ is the unique cell in $S$ that lies in the the same $\pi_{1}(M)$-orbit with $\partial_{k}c$. Then we have
$$\partial_{k}c\otimes \epsilon_{v(\partial_{k}c)}=\eta_{j} d\otimes \epsilon_{v_{i}},$$
in $\partial B_{\bullet}$, where $\eta_{i}=\pm 1$.
\end{corollary}
\begin{proof}
It is a direct consequence of Equation \eqref{eqn:transporting}.
\end{proof}
\begin{proposition}\label{prop:acyclicB}
The based chain complex $\mathcal{B}_{\bullet}$ is acyclic. In addition,
$$\mathrm{tor}(\mathcal{B}_{\bullet})=\pm 1.$$
\end{proposition}
\begin{proof}
By Corollary \ref{cor:equivariantB}, we have the following up to sign isomorphism of based chain complexes
$$\mathcal{B}_{\bullet}=\bigoplus_{i\in V}K_{\bullet}(\partial_{i}M),$$
where $K_{\bullet}(\partial_{i}M)$ is the twisted CW chain complex of $\partial_{i}M$, with basis chosen to be cells in $S$. Here by up to sign, we mean the basis of $\mathcal{B}_{\bullet}$ is sent to basis of $\bigoplus_{i\in V}K_{\bullet}(\partial_{i}M)$ up to sign. Hence it suffices to show each summand $K_{\bullet}(\partial_{i}M)$ is acyclic, and has torsion $\pm1$. We prove this for both type (4, 2) and type (3, 3).

Suppose $\partial_{i}M$ admits a celullar fundamental domain $D$ of type $(4, 2)$. Then we have 
$$\rho(\pi_{1}([m]))\epsilon_{v}=x_{m}^{2}\epsilon_{v} \quad \text{and}\quad\rho(\pi_{1}([l]))\epsilon_{v}=x_{l}^{2}\epsilon_{v},$$ where
$$x_{m}=\chi_{\rho}([m])\quad \text{and,} \quad x_{l}=\chi_{\rho}({l}).$$ 
The nontrivial assumption of $\chi_{\rho}$ implies $x_{m}^2\ne 1$ or $x_{l}^2\ne 1$. Without loss of generality, we may assume $x_{l}^2\ne 1$. Retracting the CW chain complex, we may assume the chain complex $K_{\bullet}(\partial_{i}M)$ contains a single $2$-cell $t$, two $1$-cells $m$ and $l$, and a single $0$-cell $\star$. Then the chain complex is written as
$$\langle t\otimes \epsilon_{v}\rangle\xrightarrow{d_{2}}\langle m\otimes \epsilon_{v}, l\otimes \epsilon_{v}\rangle\xrightarrow{d_{1}} \langle \star \otimes \epsilon_{v}\rangle$$
where 
$$d_{2}(t\otimes \epsilon_{v})=m\otimes \epsilon_{v}+l\otimes \epsilon_{v}-m\otimes x_{l}^2\epsilon_{v}-l\otimes x_{m}^{-2}\epsilon_{v}$$ 
and 
$$d_{1}(m)=\star\otimes \epsilon_{v}-\star\otimes x_{m}^{-2}\epsilon_{v},\quad d_{1}(l)=\star\otimes x_{l}^2\epsilon_{v}-\star\otimes \epsilon_{v}.$$
The chain complex is clearly acyclic. It remains to compute its torsion, we lift boundary element $\star$ to $(x_{l}^2-1)^{-1}l\otimes \epsilon_{v}$, and choose boundary element $d_{2}(t\otimes \epsilon_{v})$. By the definition of torsion, we have
\begin{align*}
\mathrm{tor}(K_{\bullet}(\partial_{i}M))=\det \begin{bmatrix}
   1-x_{l}^2 & 0\\
   1-x_{m}^{-2} &(x_{l}^2-1)^{-1}
\end{bmatrix}=-1.
\end{align*}
This proves the claim.

Suppose $\partial_{i}M$ admits a cellular fundamental domain $D$ of type $(3, 3)$. Then we have
$$\rho(\pi_{1}([a\circ m]))\epsilon_{v}=x_{ma}^2\epsilon_{v},\quad  \rho(\pi_{1}([l\circ a]))\epsilon_{v}=x_{al}^2\epsilon_{v},\hspace{3mm}\text{and}\quad \rho(\pi_{1}([m^{-1}\circ l]))\epsilon_{v}=x_{lm^{-1}}^2\epsilon_{v}$$ where
$$x_{ma}=\chi_{\rho}([a\circ m]),\quad \text{}\quad x_{al}=\chi_{\rho}([l\circ a]) \quad \text{and,}\quad x_{lm^{-1}}=\chi_{\rho}([m^{-1}\circ l]).$$
The nontrivial assumption of $\chi_{\rho}$ implies one of $x_{ma}^2, x_{al}^2$ and $x_{lm^{-1}}^2$ not equal to $1$. Without loss of generality, we may assume $x_{al}^2\ne 1$. Retracting the CW chain complex, we may assume the chain complex $K_{\bullet}(\partial_{i}M)$ contains a single $2$-cell $t$, three $1$-cells $m$, $a$ and $l$, and two $0$-cells $\star$ and $\star^{*}$. The chain complex is written as
$$\langle t\otimes \epsilon_{v}\rangle\xrightarrow{d_{2}}\langle a\otimes \epsilon_{v}, m\otimes \epsilon_{v}, l\otimes \epsilon_{v}\rangle\xrightarrow{d_{1}} \langle \star \otimes \epsilon_{v}, \star' \otimes \epsilon_{v}\rangle$$
where 
$$d_{2}(t\otimes \epsilon_{v})=m\otimes \epsilon_{v}+a\otimes \epsilon_{v}+l\otimes \epsilon_{v}-m\otimes x_{al}^2\epsilon_{v}-a\otimes x_{lm^{-1}}^2\epsilon_{v}-l\otimes x_{ma}^{-2}\epsilon_{v}$$ 
and 
$$d_{1}(m)=\star\otimes \epsilon_{v}-\star'\otimes x_{ma}^{-2}\epsilon_{v},\quad d_{1}(l)=\star\otimes x_{al}^2\epsilon_{v}-\star'\otimes \epsilon_{v},\quad d_{1}(a)=\star'\otimes \epsilon_{v}-\star\otimes\epsilon_{v}$$
The chain complex is clearly acyclic. It remains to compute its torsion. We lift the boundary elements $\star$ and $\star'$ to
$$(x_{al}^2-1)^{-1}(a + l)\quad \text{and }\quad a+(x_{al}^2-1)^{-1}(a + l)$$
respectively and choose the boundary element $d_{2}(T\otimes \epsilon_{v})$. The torsion is computed explicity as follows:
\begin{align*}
\mathrm{tor}(K_{\bullet}(T_{i}))=\det \begin{bmatrix}
   1-x_{lm^{-1}}^2 & \frac{1}{x_{al}^2-1} & \frac{x_{al}^2}{x_{al}^2-1}\\
   1-x_{al}^2 &0 & 0\\
   1-x_{ma}^{-2} &\frac{1}{x_{al}^2-1} &\frac{1}{x_{al}^2-1}\\
\end{bmatrix}=-1.
\end{align*}
This proves the claim.
\end{proof}

Now we define $\mathcal{K}_{\bullet}$ to be the mapping cone of $(\mathcal{A}_{\bullet}\xrightarrow{ev}\mathcal{B}_{\bullet-1})$. Precisely, $\mathcal{K}_{i}=\mathcal{A}_{i}\oplus\mathcal{B}_{i}$ and its differential map $d_{\mathcal{K}}$ is defined as follows:
\begin{equation*}
d_{\mathcal{K}}=\begin{bmatrix}
    -d_{\mathcal{A}} & 0 \\
    \mathrm{ev} & d_{\mathcal{B}}
\end{bmatrix}
\end{equation*}

The chosen bases of $\mathcal{A}_{\bullet}$ and $\mathcal{B}_{\bullet}$  naturally makes $\mathcal{K}_{\bullet}$ a based chain complex. We call the basis obtained this way the topological basis of $\mathcal{K}_{\bullet}$. Applying Milnor's multiplicativity Theorem \ref{thm:milnor} to the short exact sequence of the based chain complexes
$$0\to \mathcal{B}_{\bullet}\to \mathcal{K}_{\bullet}\to \mathcal{A}_{\bullet}\to 0,$$
one obtains
$$\mathrm{tor}(\mathcal{A}_{\bullet})=\pm\mathrm{tor}(\mathcal{K}_{\bullet})\mathrm{tor}(\mathcal{B}_{\bullet})^{-1}\mathrm{tor}(\mathcal{H}_{\bullet})^{-1}.$$
By Proposition \ref{prop:acyclicB}, $\mathrm{tor}(\mathcal{B_{\bullet}})=1$. It remains to compute $\mathrm{tor}(\mathcal{K}_{\bullet})$ and $\mathrm{tor}(\mathcal{H}_{\bullet})$. Again by Proposition \ref{prop:acyclicB} again, the long exact sequence of homology homology groups contains two non-zero maps:
\begin{equation}\label{eq:homologyKA}
H_{1}(\mathcal{K}_{\bullet})\cong H_{1}(\mathcal{A}_{\bullet}), \quad H_{2}(\mathcal{K}_{\bullet})\cong H_{2}(\mathcal{A}_{\bullet}).\end{equation}
Identifying the homology groups of $\mathcal{K}_{\bullet}$ and $\mathcal{A}_{\bullet}$ using the isomorphism above, we have
\begin{equation}\label{eq:applyM}
\mathrm{tor}(M, \partial M; \boldsymbol h_{1}, \boldsymbol h_{2})=\pm\mathrm{tor}(\mathcal{K}_{\bullet};\boldsymbol h_{1}, \boldsymbol h_{2}).
\end{equation}

\subsection{Torsion of Mapping Cone $\mathcal{K}_{\bullet}$}
In this subsection, we compute 
$$\mathrm{tor}(\mathcal{K}_{\bullet};\boldsymbol h_{\bullet}^{\boldsymbol \gamma}),$$
where $\boldsymbol h_{\bullet}^{\boldsymbol \gamma}$ denotes bases of homology groups of $\mathcal{K}_{\bullet}$ that are dual to Porti's basis $\{\boldsymbol{h}_{(M, \boldsymbol \gamma)}^{1}, \boldsymbol{h}_{M}^{2}\}$ under the identification in \eqref{eq:homologyKA}.
\begin{enumerate}
    \item [Step 1.] We first introduce a geometric basis $\boldsymbol{G}=\{G_{3},\ldots, G_{0}\}$ for $\mathcal{K}_{\bullet}$ and show that changing from the topological basis to the geometric basis changes the torsion of $\mathcal{K}_{\bullet}$ by a factor equal 
    $$2^{|T|}\prod_{\Delta\in T}^{N}z_{\Delta}^{f_{\Delta}''}(z_{\Delta}'')^{-f_{\Delta}},$$
    where $f_{\Delta}, f_{\Delta}'$ and $f_{\Delta}''$ form a strong combinatorial flattening, as defined in \ref{sec:2.3}.
    \item [Step 2.] With $\mathcal{K}_{\bullet}$ equipped with the geometric basis $\boldsymbol{G}$, the matrix representations of its boundary maps contain identity blocks, so we apply Lemma \ref{lem:retraction} and retracts $\mathcal{K}_{\bullet}$ to
    $$\mathbb{C}^{T}\xrightarrow{D} \mathbb C^{E}\to 0.$$
    By exlicit computation, we show in Proposition \ref{prop:matrixK} that the matrix representation of $D$ with respect to the geometric basis equal to $\frac{1}{2}(A\Delta_{z''}+B\Delta_{z}^{-1})$.
    \item [Step 3.] By explicit computation, we describe the bilinear forms
    $$\mathrm{coker}(D)\times H_{2}(M;\mathrm{Ad}\rho)\to \mathbb{C},$$
    $$\mathrm{ker}(D)\times H_{1}(M; \mathrm{Ad}\rho)\to \mathbb{C},$$
    in Proposition \ref{prop:pairing1} and \ref{prop:pairing2}. Then we use a linear algebra lemma \ref{lem:bordered} to write
    $$\mathrm{tor}(\mathbb{C}^{T}\xrightarrow{D} \mathbb C^{E}\to 0; \boldsymbol{h}_{\bullet}^{\gamma})=\frac{2^{-|E|}}{\det(C^TC)}
 \det\begin{pmatrix}
 K&C\\
 L_{\boldsymbol\gamma}&0
 \end{pmatrix}.$$
\end{enumerate}
\subsubsection{Geometric basis}
We introduce the following convention on the lifted vertices of cells of $\mathcal{T}$. Let $v_{i}: \bar T\to \mathbb{C}^{2}$ be the map sending a tetrahedron $\Delta$ to its lifted $i$-th vertex in $\mathbb{C}^{2}$. Here $i$ takes value in $\{0, 1, 2, 3\}$. We also define such maps over the set of faces $\bar F$ and the set of edges $\bar E$ of $\mathcal{T}$, and still denote them by $v_{i}$, where $i$ takes value in $\{0, 1, 2\}$ and $\{0, 1\}$, respectively. We will also use the shorthand notation
$$[ij]_{c}:=\det(v_{i}(c), v_{j}(c)),$$
where $c$ takes value in the set of cells of $\bar E, \bar F$ and $\bar T$.
\subsubsection{Geometric basis for $\mathcal{K}_{0}$ and $\mathcal{K}_{1}$}
For each edge $e$ of $E$, we have the following summand of $\mathcal{A}_{1}\to \mathcal{B}_{0}$.
$$ev:\langle e\otimes \mathfrak e, e\otimes \mathfrak f, e\otimes \mathfrak h \rangle\to \langle \partial_{0}e\otimes \epsilon_{v_{0}(e)}, \partial_{1}e\otimes \epsilon_{v_{1}(e)} \rangle.$$
\begin{proposition}\label{prop:edgedet}
For each edge $e\in \bar E$, define
$$M_{X}(e):=[v_{0}(e), v_{1}(e)]X[v_{0}(e), v_{1}(e)]^{-1}$$
for $X\in \mathfrak{sl}_{2}(\mathbb C)$. The evaluation map has the following properties, 
\begin{enumerate}[(1)]
\item Its kernel is generated by the vector
$$e\otimes M_{\mathfrak h}(e).$$
\item Vectors
$$[01]_{e}^{-1}e\otimes M_{\mathfrak e}(e),\text{ and } -[01]_{e}^{-1}e\otimes M_{\mathfrak f}(e)$$
are lifts of $\partial_{1}e\otimes \epsilon_{v_{1}(e)}$ and $\partial_{0}e\otimes \epsilon_{v_{0}(e)}$ respectively.
\item The determinant
$$[[01]_{e}^{-1}e\otimes M_{\mathfrak e}(e), [01]_{e}^{-1}e\otimes M_{\mathfrak f}(e), e\otimes M_{\mathfrak h}(e); e\otimes \mathfrak e, e\otimes \mathfrak f, e\otimes \mathfrak h]=\pm[01]^{-2}_{e}$$
\end{enumerate}
\end{proposition}
\begin{proof}
(2) implies the kernel is one-dimensional. Hence we first check (2), then (1). By the definition of the evaluation map,
\begin{align*}
ev(e\otimes M_{\mathfrak e}(e))&=\partial_{0}e\otimes \det(M_{\mathfrak e}(e)v_{0}, v_{0})\epsilon_{v_{0}(e)}+\partial_{1}e\otimes\det(M_{\mathfrak e}(e)v_{1}, v_{1})\epsilon_{v_{1}(e)}\\
&=\det(0, v_{0})\partial_{0}e\otimes\epsilon_{v_{0}(e)}+\det(v_{0}, v_{1})\partial_{1}e\otimes\epsilon_{v_{1}(e)}\\
&=[01]_{e}\partial_{1}e\otimes\epsilon_{v_{1}(e)}.
\end{align*}
Here the second equality follows from 
$$(v_{0}, v_{1})^{-1}v_{0}=\begin{bmatrix}
1 \\
0
\end{bmatrix}, \quad (v_{0}, v_{1})^{-1}v_{1}=\begin{bmatrix}
0 \\
1
\end{bmatrix},$$
for any pair of vectors $v_{0}, v_{1}$ spanning $\mathbb C^{2}$.
By a similar computation, we obtain
$$ev(e\otimes M_{\mathfrak f}(e))=-[01]_{e}\partial_{0}e\otimes \epsilon_{v_{0}(e)}.$$
This proves (2). Now compute
\begin{align*}
ev(e\otimes M_{\mathfrak h}(e))=\det(v_{0}, v_{0})\partial_{0}e\otimes\epsilon_{v_{0}(e)}+\det(-v_{1}, v_{1})\partial_{1}e\otimes\epsilon_{v_{1}(e)}.
\end{align*}
This proves (1). Statement (3) follows from the fact
$$[g\mathfrak eg^{-1}, g\mathfrak fg^{-1}, g\mathfrak hg^{-1}; \mathfrak e, \mathfrak f, \mathfrak h]=\pm 1$$
for all $g\in GL(2, \mathbb C)$.
\end{proof}
For each face $f\in \bar{F}$, we have the following summand
$$ev:\langle f\otimes \mathfrak e, f\otimes \mathfrak f, f\otimes \mathfrak h \rangle\to \langle \partial_{i}f\otimes \epsilon_{v_{i}(f)}|i\in \{0, 1, 2\} \rangle$$
of $\mathcal{A}_{2}\to \mathcal{B}_{1}$.
\begin{proposition}\label{prop:facedet}
For each face $f\in F$, the map $ev(f)$ is an isomorphism with determinant
$$[ev(f)\otimes \mathfrak e, ev(f)\otimes \mathfrak f, ev(f)\otimes \mathfrak h; \partial_{0}f\otimes \epsilon_{v_{0}(f)}, \partial_{1}f\otimes \epsilon_{v_{1}(f)}, \partial_{2}f\otimes \epsilon_{v_{2}(f)}]=\pm 2[01]_{f}[02]_{f}[12]_{f}.$$
\end{proposition}
\begin{proof}
Assume
$$v_{0}=\begin{bmatrix}
x_{0}\\
y_{0}
\end{bmatrix}, \quad v_{1}=\begin{bmatrix}
x_{1}\\
y_{1}
\end{bmatrix}, \quad v_{2}=\begin{bmatrix}
x_{2}\\
y_{2}
\end{bmatrix}.$$
The determinant equals
\begin{align*}
&\det\begin{bmatrix}
\det({\mathfrak e v_{0}, v_{0}}) & \det({\mathfrak f v_{0}, v_{0}}) & \det({\mathfrak h v_{0}, v_{0}})\\
\det({\mathfrak e v_{1}, v_{1}}) & \det({\mathfrak f v_{1}, v_{1}}) & \det({\mathfrak h v_{1}, v_{1}})\\
\det({\mathfrak e v_{2}, v_{2}}) & \det({\mathfrak f v_{2}, v_{2}}) & \det({\mathfrak h v_{2}, v_{2}})\\
\end{bmatrix}\\
=&\det\begin{bmatrix}
y_{0}^{2} & -x_{0}^{2} & 2x_{0}y_{0}\\
y_{1}^2 & -x_{1}^2& 2x_{1}y_{1}\\
y_{2}^2 & -x_{2}^2 & 2x_{2}y_{2}
\end{bmatrix}\\
=&-2(x_{0}y_{1}-x_{1}y_{0})(x_{0}y_{2}-x_{2}y_{0})(x_{1}y_{2}-x_{2}y_{1})=-2[01]_{f}[02]_{f}[12]_{f}.
\end{align*}
This proves the identity.
\end{proof}
We define the geometric basis for $\mathcal{K}_{0}=\mathcal{B}_{0}$ to be
$$G_{0}=\{\partial_{0}e, \partial_{1}e|e\in E\},$$
and the geometric basis for $\mathcal{K}_{1}=\mathcal{A}_{1}\oplus \mathcal{B}_{1}$ to be
$$G_{1}:=\{[01]_{e}^{-1}e\otimes M_{\mathfrak e}(e), -[01]_{e}^{-1}e\otimes M_{\mathfrak f}(e), e\otimes M_{\mathfrak h}(e)|e\in E\}\cup \{ev(f)\otimes \mathfrak e, ev(f)\otimes \mathfrak f, ev(f)\otimes \mathfrak h|f\in F\}.$$
Propositions \ref{prop:edgedet} and \ref{prop:facedet} imply $G_{1}$ is a basis of $\mathcal{K}_{1}$.

\subsubsection{Geometric basis for $\mathcal{K}_{2}$ and $\mathcal{K}_{3}$}
For each tetrahedron $\Delta\in \bar{T}$, we have the following summand 
$$ev:\langle \Delta\otimes \mathfrak e, \Delta\otimes \mathfrak f, \Delta\otimes \mathfrak h\rangle\to \langle \partial_{i}\Delta\otimes \epsilon_{v_{i}(\Delta)} |i\in \{0, 1, 2, 3\}\rangle$$
of $ev_{3}: \mathcal{A}_{3}\to \mathcal{B}_{2}$. We have the following proposition.
\begin{proposition}\label{prop:crossratio}
The cross-ratio function
\begin{align*}
\mathrm{CR}:(\mathbb{C}^{2})^{4}&\to \mathbb{C}\\
(v_{0}, v_{1}, v_{2}, v_{3})&\mapsto -\frac{\det(v_{0}, v_{3})\det(v_{1}, v_{2})}{\det(v_{0}, v_{1})\det(v_{2}, v_{3})}.
\end{align*}
descends to a function $(\mathbb{C}P^{1})^{4}\to \mathbb{C}$, the cotangent vector $d\log \mathrm{CR}|_{(v_{0}(\Delta),\ldots, v_{3}(\Delta))}\ne 0$ and vanishes on the image of $ev(\Delta)$.
\end{proposition}
\begin{proof}
Since each of $v_{0}, v_{1}, v_{2}, v_{3}$ appears in the numerator and the denominator exactly once, the value of $\mathrm{CR}$ does not depend on the lifting from $\mathbb{C}P^{1}$ to $\mathbb{C}^{2}$. Hence $\mathrm{CR}$ defines a function on $(\mathbb CP^{1})^{4}$.

Then we notice $\mathrm{CR}$, hence $\log \mathrm{CR}$, is invariant under the action of $PSL(2, \mathbb C)$. To show 
$$d\log \mathrm{CR}|_{(v_{0}(\Delta),\ldots, v_{3}(\Delta))}\ne 0,$$ we may assume $\{v_{0}(\Delta), v_{1}(\Delta), v_{2}(\Delta), v_{3}(\Delta)\}=\{ \infty, 0, 1, z\}$. By direct computation,
$$d\log \mathrm{CR}=d\log\left(\frac{1}{1-z}\right)$$
which is nonzero if $z\notin \{0, 1, \infty\}$.

It remains to prove $d\log \mathrm{CR}$ vanishes on the image of $ev(\Delta)$. Since the tangent vectors in $ev(\Delta)$ comes from the infinitesimal action of the Lie algebra of $PSL(2, \mathbb{C})$, the derivative of $\mathrm{CR}$ vanishes in those directions by the $PSL(2, \mathbb C)$-invariance of $\mathrm{CR}$.
\end{proof}

Now for each tetrahedron $\Delta\in \bar T$, we define $s_{\Delta}\in \langle \partial_{i}\Delta\otimes \epsilon_{v_{i}(\Delta)} |i\in \{0, 1, 2, 3\}\rangle$ such that 
\begin{equation}\label{eq:defineS}
d\log \mathrm{CR}|_{(v_{0}(\Delta),\ldots, v_{3}(\Delta))}(s_{\Delta})=1.   
\end{equation}

\begin{proposition}\label{prop:tetdet}
For each tetrahedron $\Delta\in \bar T$, and any choice of $s_{\Delta}$ satisfying \eqref{eq:defineS}, we have
$$[ev(\Delta)\otimes \mathfrak e, ev(\Delta)\otimes \mathfrak f, ev(\Delta)\otimes \mathfrak h, s_{\Delta};\partial_{0}\Delta\times \epsilon_{v_{0}(\Delta)},\ldots,\partial_{3}\Delta\times \epsilon_{v_{3}(\Delta)}]=\pm {2}{[03]_{\Delta}[12]_{\Delta}[01]_{\Delta}[23]_{\Delta}}.$$
\end{proposition}
\begin{proof}
By Proposition \ref{prop:crossratio}, $s_{\Delta}$ is uniquely defined up to the kernel of $d\log \mathrm{CR}$, which is equal to the image of $ev(\Delta)$. Hence, change $s_{\Delta}$ by any vectors in $\mathrm{Image}(ev(\Delta))$ does not change the determinant.

We claim $-\frac{[03]_{\Delta}[01]_{\Delta}}{[13]_{\Delta}}\epsilon_{v_{0}(\Delta)}$ is a choice of $s_{\Delta}$. Notice this tangent vector describes the action:
$$A(t):(v_{0}, v_{1}, v_{2}, v_{3})\mapsto (v_{0}+\frac{[03]_{\Delta}}{[13]_{\Delta}}tv_{1}, v_{1}, v_{2}, v_{3}).$$
The derivative
$$\frac{d\log \mathrm{CR}(A(t)(v_{0},\ldots,v_{3}))}{dt}=\frac{[13]_{\Delta}}{[03]_{\Delta}}\frac{[03]_{\Delta}}{[13]_{\Delta}}-\frac{[11]_{\Delta}}{[01]_{\Delta}}\frac{[03]_{\Delta}}{[13]_{\Delta}}=1.$$
This verifies the claim.
Now we compute the determinant:
\begin{align*}
&\det\begin{bmatrix}
\det({\mathfrak e v_{0}, v_{0}}) & \det({\mathfrak fv_{0}, v_{0}}) & \det({\mathfrak hv_{0}, v_{0}}) & -\frac{[03]_{\Delta}[01]_{\Delta}}{[13]_{\Delta}}\\
\det({\mathfrak ev_{1}, v_{1}}) & \det({\mathfrak f v_{1}, v_{1}}) & \det({\mathfrak hv_{1}, v_{1}}) & 0\\
\det({\mathfrak ev_{2}, v_{2}}) & \det({\mathfrak f v_{2}, v_{2}}) & \det({\mathfrak hv_{2}, v_{2}}) & 0\\
\det({\mathfrak ev_{3}, v_{3}}) & \det({\mathfrak f v_{3}, v_{3}}) & \det({\mathfrak hv_{3}, v_{3}}) & 0\\
\end{bmatrix}\\
=&-\frac{[03]_{\Delta}[01]_{\Delta}}{[13]_{\Delta}}\big(2[12]_{\Delta}[13]_{\Delta}[23]_{\Delta}\big)\\
=& -2[03]_{\Delta}[12]_{\Delta}[01]_{\Delta}[23]_{\Delta}.
\end{align*}
This proves the proposition.
\end{proof}
We define the geometric basis for $\mathcal{K}_{2}=\mathcal{A}_{2}\oplus \mathcal{B}_{2}$ to be
$$G_{2}:=\{f\otimes \mathfrak e, f\otimes \mathfrak f, f\otimes \mathfrak h|f\in \bar F\}\cup\{ev(\Delta\otimes \mathfrak e), ev(\Delta\otimes \mathfrak f), ev(\Delta\otimes \mathfrak h), s_{\Delta}|\Delta\in \bar T\}$$
and the geometric basis for $\mathcal{K}_{3}=\mathcal{A}_{3}$ to be
$$G_{3}:=\{\Delta\otimes \mathfrak{e}, \Delta\otimes \mathfrak{f}, \Delta\otimes \mathfrak{h}|\Delta\in \bar T\}.$$
Proposition \ref{prop:tetdet} implies $G_{2}$ is a basis.
\begin{remark}
Although tangent vectors $\epsilon_{v}$ depend on the lift of vertex $v$. The vectors $s_{\Delta}$ and $\kappa_{e}$ are independent of the lifting. Therefore, $s_{\Delta}$ and $\kappa_{e}$ are actually defined for $\Delta\in T$ and $e\in E$ respectively.
\end{remark}
\subsubsection{Flattening monomial as determinant of changing basis}
We denote the chain complex $\mathcal{K}_{\bullet}$ equipped with the geometric basis $\boldsymbol{G}=\{G_{3},\ldots, G_{0}\}$ by $\mathcal{K}_{\bullet}^{\boldsymbol{G}}$, The flattening monomial in \eqref{eq:main1} naturally appears in the determinant of the change from the topological basis to the geometric basis.
\begin{proposition}\label{prop:flat}
Let $f_{\Delta}, f_{\Delta}', f_{\Delta}''$ be a strong flattening of $\mathcal{T}$. Then
$$\mathrm{tor}(\mathcal{K}_{\bullet};\boldsymbol{h}_{\bullet})=\pm 2^{|T|}\prod_{\Delta\in T}z_{\Delta}^{f_{\Delta}''}(z_{\Delta}'')^{-f_{\Delta}}\mathrm{tor}(\mathcal{K}_{\bullet}^{\boldsymbol{G}}; \boldsymbol{h}_{\bullet}).$$
for any choice of basis $\boldsymbol{h}_{\bullet}$ of the homology groups.
\end{proposition}
\begin{proof}
Proposition \ref{prop:edgedet} and \ref{prop:facedet} computed the determinant of changing the topological basis to geometric basis of $\mathcal{K}_{1}$, and $\ref{prop:tetdet}$ compute the determinant of changing the topological basis to geometric basis of $\mathcal{K}_{2}$. The topological basis and the geometric basis are identical of $\mathcal{K}_{3}$ and $\mathcal{K}_{0}$. Putting them together, we have
\begin{equation}\label{eq:gtcompare}
\frac{\mathrm{tor}(\mathcal{K}_{\bullet}^{G};\boldsymbol{h}_{\bullet})}{\mathrm{tor}(\mathcal{K}_{\bullet};\boldsymbol{h}_{\bullet})}=\prod_{\Delta\in \bar T}(2[03]_{\Delta}[12]_{\Delta}[01]_{\Delta}[23]_{\Delta})\prod_{f\in \bar F}\frac{1}{2[01]_{f}[02]_{f}[12]_{f}}\prod_{e\in \bar E}[01]_{e}^{2}
\end{equation}
We need the following lemma to compare the determinant factors coming from cells of different dimensions.
\begin{lemma}\label{lem:transportface} Let $\Delta\in \bar T$ be a tetrahedra and let $f\in \bar F$ be a face that are identified with a face $\hat f$ of $\Delta$. Denote $a:\{0, 1, 2\}\to \{0, 1, 2, 3\}$ the identification of vertices of $f$ and $\Delta$. Let $g_{\Delta}^{i}$ be the unique element sending some cell in $S$ to the $i$-th corner $\partial_{i}\Delta$. We have
    $$\frac{[a(i)a(j)]_{\Delta}}{[ij]_{f}}=\pm{\chi_{\rho}(d_{i})\chi_{\rho}(d_{j})}$$
 where $\chi_{\rho}$ is the induced character on $\pi_{1}(\partial M)$, and $d_{i}$ is the unique deck transformation taking $\partial_{a(i)}(g_{\Delta}^{a(i)}\hat f)$ to a cell belongs to $S$.
\end{lemma}
\begin{proof}
Let $g_{f}^{i}$ be the group element in $\pi_{1}(M)$ sending a cell belongs to $S$ to the $i$-th corner of $f$, and $g$ be the unique group element in $\pi_{1}(M)$ identifies $f$ with a face of $\Delta$. We have the following relations between these group elements:
$$g_{\Delta}^{a(i)}=g\circ g_{f}^{i}\circ d_{i}$$
$$g_{\Delta}^{a(j)}=g\circ g_{f}^{j}\circ d_{j}$$
Suppose the $i$-th and $j$-th vertices of face $f$ end at the $m$-th and $n$-th boundary component of $M$. We have
\begin{align*}
\frac{[a(i)a(j)]_{\Delta}}{[ij]_{f}}=\frac{\det(\rho(g_{\Delta}^{a(i)})v_{m}, \rho(g_{\Delta}^{a(j)}v_{n})}{\det(\rho(g_{f}^{i})v_{m}, \rho(g_{f}^{j})v_{n})}=\frac{\det(\rho(g\circ g_{f}^{i}\circ d_{i}) v_{m}, \rho(g\circ g_{f}^{j}\circ d_{j}) v_{n})}{\det(\rho(g_{f}^{i})v_{m}, \rho(g_{f}^{i})v_{n})}=\pm{\chi_{\rho}(d_{i})\chi_{\rho}(d_{j})}
\end{align*}
The firse quality is the definition, and the third equality follows from $v_{m}$ and $v_{n}$ are eigenvectors of $d_{i}$ and $d_{j}$ respectively. Using identical argument, we also have the following lemma for the edges.
\end{proof}
\begin{lemma}\label{lem:transportedge} Let $\Delta\in \bar T$ be a tetrahedra and let $e\in \bar E$ be a face that are identified with a face $\hat e$ of $\Delta$. Denote $a_{e, \Delta}:\{0, 1 \}\to \{0, 1, 2, 3\}$ the identification of vertices of $f$ and $\Delta$. Let $g_{\Delta}^{i}$ be the unique element sending the $i$-th corner $\partial_{i}\Delta$ to $S$. We have
    $$\frac{[a_{e, \Delta}(i)a_{e, \Delta}(j)]_{\Delta}}{[ij]_{e}}=\chi_{\rho}(d_{i})\chi_{\rho}(d_{j})$$
    where $\chi_{\rho}$ is the induced character on $\pi_{1}(\partial M)$, and $d_{i}$ is the unique deck transformation taking $\partial_{a(i)}(g_{\Delta}^{a(i)}\hat e)$ to a cell belongs to $S$.
\end{lemma}
Notice that each face of $\bar{F}$ is attached to faces of two tetrahedra. If we transport corners of each faces of each tetrahedra in $\bar{T}$, by element in $\pi_{1}(M)$ that transport the corner of the tetrahedron it belongs to to $S$. We obtains all the inner $1$-cells of $D_{i}$, counted twice; and all the boundary $1$-cells of $D_{i}$, counted once. Applying Lemma \ref{lem:transportface}, we obtains
\begin{equation}
\frac{\mathrm{tor}(\mathcal{K}_{\bullet}^{G};\boldsymbol{h}_{\bullet})}{\mathrm{tor}(\mathcal{K}_{\bullet};\boldsymbol{h}_{\bullet})}=2^{-|T|}\prod_{\Delta\in {T}}\frac{1}{[02]_{\Delta}[13]_{\Delta}}\prod_{e\in E}[01]_{e}^{2}\prod_{i\in V}M_{i}
\end{equation}
where 
$$M_{i}=x_l^{|m'|}x_{m}^{-|l'|},$$
if the fundamental domain $D_{i}$ of $\partial_{i}M$ is of type (4,2), and
$$M_{i}=x_{al}^{|m'|}x_{ma}^{-|l'|}x_{lm^{-1}}^{|a'|}$$
if the fundamental domain $D_{i}$ of $\partial_{i}M$ is of type (3,3). Here $|s|$ is the number of $1$-cells in the path $s$, and $x_{m}, x_{l}, x_{al}, x_{ma}$ and $x_{lm^{-1}}$ are the same as those in the proof of Proposition \ref{prop:acyclicB}.

Applying the combinatorial flattening condition that the sum of numbers around an edge equals to $2$ together with Lemma \ref{lem:transportedge} we have
\begin{equation}\label{eq:distributeflat}
\begin{aligned}
    \prod_{\Delta\in {T}}\frac{1}{[02]_{\Delta}[13]_{\Delta}}\prod_{e\in E}[01]_{e}^{2}&=\prod_{\Delta\in {T}}(\frac{[01][23]}{[02][13]})^{f_{\Delta}}(\frac{[02][13]}{[02][13]})^{f_{\Delta}'}(\frac{[03][12]}{[02][13]})^{f_{\Delta}''}\prod_{i\in V}N_{i}\\
&=\prod_{\Delta\in {T}}(z_{\Delta}'')^{f_{\Delta}}z_{\Delta}^{-f_{\Delta}''}\prod_{i\in V}N_{i}
\end{aligned}
\end{equation}
where 
$$N_{i}=x_{l}^{-f(m')}x_{m}^{f(l')}$$
if the fundamental domain $D_{i}$ of $\partial_{i}M$ is of type (4,2). Here $f(m_{i}')$ and $f(l_{i}')$ equal the sum of the flattening numbers over the corners in $D_{i}$ adjacent to $0$-cells of $m_{i}$ and $l_{i}$ respectively, and 
$$N_{i}=x_{al}^{-f(m')}x_{ma}^{f(l')}x_{lm^{-1}}^{-f(a')}.$$
if the fundamental domain $D_{i}$ of $\partial_{i}M$ is of type (4,2). Here $f(m_{i}')$ and $f(l_{i}')$ equal the sum of the flattening numbers over the corners in $D_{i}$ adjacent to $0$-cells in $\partial D_{i}\setminus \{a_{i}'\}$ of $m_{i}$ and $l_{i}$ respectively, and $f(a_{i}')$ equal the sum of the flattening numbers over the corners in $D_{i}$ adjacent to $0$-cells in $a_{i}'$.

It remains to show $M_{i}N_{i}=1$. For case $(4, 2)$, we have
\begin{equation}\label{eq:42cancel}
M_{i}N_{i}=x_{l}^{\sum_{\Delta}G_{m, \Delta}f_{\Delta}+G_{m, \Delta}'f_{\Delta}'+G_{m, \Delta}''f_{\Delta}}x_{m}^{-\sum_{\Delta}G_{l, \Delta}f_{\Delta}-G_{l, \Delta}'f_{\Delta}'-G_{l, \Delta}''f_{\Delta}}=1
\end{equation}
where the first equality follows form $f_{\Delta}+f_{\Delta}'+f_{\Delta}''=1$, and the second equality follows from the strong flattening condition.

For case $(3, 3)$, we have
\begin{equation}\label{eq:33cancel}
M_{i}N_{i}=x_{al}^{\sum_{\Delta}G_{a\circ m, \Delta}f_{\Delta}+G_{a\circ m, \Delta}'f_{\Delta}'+G_{a\circ m, \Delta}''f_{\Delta}}x_{ma}^{-\sum_{\Delta}G_{l\circ a, \Delta}f_{\Delta}-G_{l\circ a, \Delta}'f_{\Delta}'-G_{l\circ a, \Delta}''f_{\Delta}}=1
\end{equation}
where the first equality follows form $f_{\Delta}+f_{\Delta}'+f_{\Delta}''=1$, together with $x_{lm^{-1}}x_{al}^{-1}x_{ma}=1$ and the second equality follows from the strong flattening condition.

Combining Equation \eqref{eq:gtcompare}, Equation \eqref{eq:distributeflat}, Equation \eqref{eq:42cancel} and Equation \eqref{eq:33cancel}, we proved the desired equality. This complete the proof of proposition.
\end{proof}

\begin{remark}
The Equation \eqref{eq:42cancel} and Equation \eqref{eq:33cancel} in the proof indicates the strong combinatorial flattening condition is necessary. For the discrete faithful representation $\rho_{0}$, a combinatorial flattening suffices for Proposition \ref{prop:flat}.
\end{remark}    

\subsubsection{Torsion of Retracted Complex}
For the based chain complex $\mathcal{K}_{\bullet}^{G}$, the following blocs are represented by identity matrices:
$$\mathcal{A}_{3}=\langle\Delta\otimes \mathfrak e, \Delta\otimes \mathfrak f, \Delta\otimes \mathfrak h\rangle_{\Delta\in T} \to \langle ev(\Delta\otimes \mathfrak e), ev(\Delta\otimes \mathfrak f),ev(\Delta\otimes \mathfrak h) \rangle_{\Delta\in T}\subset \mathcal{B}_{2},$$
$$\mathcal{A}_{2}=\langle f\otimes \mathfrak e, f\otimes \mathfrak f, f\otimes \mathfrak h\rangle_{f\in F}\to \mathcal{B}_{1}=\langle ev(f\otimes \mathfrak e), ev(f\otimes \mathfrak f),ev(f\otimes \mathfrak h) \rangle_{f\in F},$$
$$\mathcal{A}_{1}\supset\langle[01]_{e}^{-1} e\otimes M_{\mathfrak e}(e), -[01]_{e}^{-1}e\otimes M_{\mathfrak f}(e)\rangle_{e\in E} \to \mathcal{B}_{0}=\langle e\otimes\epsilon_{v_{1}(e)}, e\otimes\epsilon_{v_{0}(e)}\rangle_{e\in E} .$$
Applying Lemma \ref{lem:retraction}, we have
\begin{proposition}\label{prop:matrixK}
We have
$$\mathrm{tor}(\mathcal{K}_{\bullet}^{\boldsymbol{G}};\boldsymbol{h}^{\gamma}_{\bullet})=\mathrm{tor}(\mathbb{C}^{T}\xrightarrow{D} \mathbb C^{E}; \{s_{\Delta}\}_{\Delta\in T}\cup \{\kappa_{e}\}_{e\in E}, \boldsymbol{h}^{\gamma}_{\bullet})$$
where the map $K$ is the following composition
$$\mathbb C^{T}\to \mathcal{B}_{2}\to \mathcal{B}_{1}\xrightarrow{ev_{2}^{-1}}\mathcal{A}_{2}\to \mathbb{C}^{E}$$
In addition, the matrix representation of $D$ with respect to basis $s_{\Delta}$ and $\kappa_{e}$ is 
$$K=\pm \frac{1}{2}(A\Delta_{z''}+B\Delta_{z}^{-1}).$$
\end{proposition}
\begin{proof}
The description of $D$ follows by a routine diagram chase in the following commutative diagram:
$$\begin{tikzcd}
&&\mathbb{C}^{E}\arrow[d]\\
\mathcal{A}_{3} \arrow[r]\arrow[d] & \mathcal{A}_{2} \arrow[r]\arrow[d, "\cong"]& \mathcal{A}_{1}\arrow[d]\\
\mathcal{B}_{2} \arrow[r]\arrow[d]& \mathcal{B}_{1} \arrow[r]& \mathcal{B}_{0} \\
\mathbb{C}^{T}&&
\end{tikzcd}$$
with the vertical arrows in the second rows are evaluation map.
Now we compute the matrix representation of $D$. Lifting $s_{\Delta}$ to $\mathcal{B}_{2}$ as in the proof of Proposition \ref{prop:tetdet}, we obtain
$$s_{\Delta}=-\frac{[03]_{\Delta}[01]_{\Delta}}{[13]_{\Delta}}\partial_{0}\Delta\otimes \epsilon_{v_{0}(\Delta)},$$
which is then sent to
\begin{align}\label{eq:sdeltainB1}
-\frac{[03]_{\Delta}[01]_{\Delta}}{[13]_{\Delta}}(\partial_{0}f_{012}-\partial_{0}f_{013}+\partial_{0}f_{023})\otimes \epsilon_{v_{0}(\Delta)}\end{align}
in $\mathcal{B}_{1}$. To find its preimage under $ev_{2}$ in $\mathcal{A}_{2}$, we need the following lemma:
\begin{lemma}\label{lem:inverseev2}
Let $f\in F$ by any face. Then $\partial_{0}f\otimes \epsilon_{v_{0}(f)}$ is identified with 
$$-\frac{[12]_{f}}{2[02]_{f}[01]_{f}}f\otimes M_{\mathfrak h}(e_{12}(f)).$$
in $\mathcal{A}_{2}$.
\end{lemma}
\begin{proof}
Since $M_{\mathfrak h}(e_{12}(f))$ is the unique element in $\mathfrak{sl}_{2}(\mathbb{C})$ fixing vertices $v_{1}(f)$ and $v_{2}(f)$ up to a factor, we have 
$$ev(f\otimes M_{\mathfrak h}(e_{12}(f))=C(\partial_{0}f\otimes \epsilon_{v_{0}(f)})$$
for some scalar ${C}$. By the definition of the evaluation map, we have
\begin{align*}
C&=\det\begin{bmatrix}
\begin{bmatrix}
x_{1} & x_{2}\\
y_{1} & y_{2}
\end{bmatrix}
\begin{bmatrix}
1 & 0\\
0 & -1
\end{bmatrix}
\begin{bmatrix}
x_{1} & x_{2}\\
y_{1} & y_{2}
\end{bmatrix}^{-1}
\begin{bmatrix}
x_{0} \\
y_{0} 
\end{bmatrix},
\begin{bmatrix}
x_{0}\\
y_{0}
\end{bmatrix}
\end{bmatrix}\\
&=\frac{2(y_{2}x_{0}-x_{2}y_{0})(x_{1}y_{0}-y_{1}x_{0})}{x_{1}y_{2}-x_{2}y_{1}}=-\frac{2[02]_{f}[01]_{f}}{[12]_{f}}.
\end{align*}
Here we assumed
$$v_{0}=\begin{bmatrix}
x_{0}\\
y_{0}
\end{bmatrix}\quad v_{1}=\begin{bmatrix}
x_{1}\\
y_{1}
\end{bmatrix}\quad v_{2}=\begin{bmatrix}
x_{2}\\
y_{2}
\end{bmatrix}.$$
\end{proof}
Applying Lemma \ref{lem:inverseev2} to \eqref{eq:sdeltainB1}, we see that $s_{\Delta}$ is sent to
\begin{align}\label{eq:expandsdelta}
&\frac{[03]_{\Delta}[01]_{\Delta}}{[13]_{\Delta}}(\frac{[12]_{\Delta}f_{012}\otimes M_{\mathfrak h}(e_{12}(\Delta))}{2[02]_{\Delta}[01]_{\Delta}}-\frac{[13]_{\Delta}f_{013}\otimes M_{\mathfrak h}(e_{13}(\Delta))}{2[03]_{\Delta}[01]_{\Delta}}+\frac{[23]_{\Delta}f_{023}\otimes M_{\mathfrak f}(e_{23}(\Delta))}{2[02]_{\Delta}[03]_{\Delta}})\\
&=\frac{1}{2}\left(z^{-1}f_{012}\otimes M_{\mathfrak f}(e_{12}(\Delta))-f_{013}\otimes M_{\mathfrak h}(e_{13}(\Delta))+z''f_{023}\otimes M_{\mathfrak f}(e_{23}(\Delta))\right)
\end{align}
in $\mathcal{A}_{2}$. Composing with the boundary map of $\mathcal{A}_{\bullet}$, we obtain 
$$D(s_{\Delta})=\frac{1}{2}\left(z''(\kappa_{e_{01}(\Delta)}+\kappa_{e_{23}(\Delta)})-(\kappa_{e_{02}(\Delta)}+ \kappa_{e_{13}(\Delta)})+z^{-1}(\kappa_{e_{03}(\Delta)}+\kappa_{e_{12}(\Delta)})\right)$$
in $\mathcal{A}_{1}$. Using the edge-tetrahedron adjacency matrices, we find that the matrix representation of $D$ equals to
$$D|_{s_{\Delta}, \kappa_{e}}=\frac{1}{2}(G\Delta_{z''}-G'+G''\Delta_{z^{-1}})=\frac{1}{2}(A\Delta_{z''}+B\Delta_{z^{-1}})$$
\end{proof}

Since the chain complex $(\mathbb{C}^{T}\xrightarrow{D}\mathbb{C}^{E}\to 0)$ is homotopy equivalent to $\mathcal{K}_{\bullet}$, their homology groups are equal. In our incomplete setting, these homology groups are natually isomorphic to those of $\mathcal{A}_{\bullet}$. In the next lemmas, we describe the perfect pairings between
\begin{enumerate}
    \item $\ker D\cong H_{2}(\mathcal{A}_{\bullet})$ and $H_{1}(M; \mathrm{Ad}\rho)$;
    \item $\mathrm{coker}(D)\cong H_{1}(\mathcal{A}_{\bullet})$ and $H_{2}(M; \mathrm{Ad}\rho)$. 
\end{enumerate}
We first describe Porti's invariant vector $I_{v}$ at vertex $v$.
\begin{proposition}
Let $v$ be a vertex of the image of $\widetilde{M}$ in $\mathbb{H}^{3}$, and $\rho: \pi_{1}(M)\to PSL(2, \mathbb{C})$ be a character satisfying the condition of Theorem \ref{thm:mainincomplete}. Then there exists a vector $v'\in \mathbb C^{2}$ with $[v]\ne [v']$ in $\mathbb CP^{1}$, such that
\begin{equation}\label{eq:definitionIv}
\mathbf I_{v}=\frac{1}{2}[v, v']H[v, v']^{-1}
\end{equation}
up to multiplication by a scalar.
\end{proposition}
\begin{proof}
Let $\gamma$ be a simple closed curve on the boundary component corresponding to the vertex $v$. Then by the assumption of $\rho$, $\rho([\gamma])$ is hyperbolic and hence fix two boundary points of $\mathbb H^{3}$, one of them must be $v$. Suppose $v'$ is the other fixed point. Conjugating by the matrix $[v, v']$ sends the two boundary points $v$ and $v'$ to $0$ and $\infty$ in $\partial \mathbb H^{3}$. After this conjugation, the invariant vector must be a multiple of $H\in \mathfrak{sl}_{2}(\mathbb C)$. The proposition follow.
\end{proof}
In the rest of the section, we will take \eqref{eq:definitionIv} as the definition of $\mathbf I_{v}$. Recall that we choose the bilinear form
$$B(X, Y):=\mathrm{tr}(XY),$$
of $\mathfrak{sl}_{2}(\mathbb C)$ for the definition of the twisted intersection pairing between homology groups $H_{\bullet}(M, \partial M; \mathrm{Ad}\rho)$ and $H_{3-\bullet}(M; \mathrm{Ad}\rho)$.

\begin{proposition}\label{prop:pairing1}
Given $x\in \ker(D)$, lifting $x$ to $\mathcal{B}_{2}$ induces a tangent vector at $\rho$ in the deformation space $\mathcal{M}(\mathcal{T})$. The natural bilinear pairing between $\ker D\cong H_{2}(\mathcal{A}_{\bullet})$ and $H_{1}(M;\mathrm{Ad}\rho)$ is
$$[x, \gamma\times \mathbf I_{v}]=\frac{1}{2}dH(\gamma)(x)$$
where $\gamma$ is a peripheral curve at vertex $v$ and $I_{v}$ is the invariant vector associated to vertex $v$.
\end{proposition}

\begin{proof}
For a tetrahedron $\Delta\in T$ with shape parameter $z$, $\partial \Delta$ consists four triangles at vertices $v_{0}(\Delta)$, $v_{1}(\Delta)$, $v_{2}(\Delta)$ and $v_{3}(\Delta)$ as follows:
$$\begin{array}{cccc}
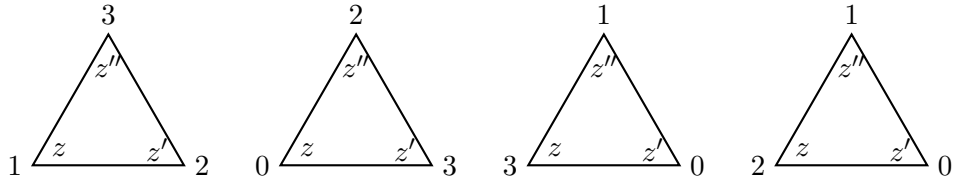
\begin{tikzpicture}
\draw[thick] (0,0) node[xshift=10, yshift=6]{$z$}node[left]{$1$}-- (2,0)node[xshift=-10, yshift=6]{$z'$}node[right]{$2$} -- (1,{2*sin(60)})node[yshift=-12]{$z''$}node[above]{$3$} -- cycle;
\end{tikzpicture}& \begin{tikzpicture}
\draw[thick] (0,0) node[xshift=10, yshift=6]{$z$}node[left]{$0$}-- (2,0)node[xshift=-10, yshift=6]{$z'$}node[right]{$3$} -- (1,{2*sin(60)})node[yshift=-12]{$z''$}node[above]{$2$} -- cycle;
\end{tikzpicture}&\begin{tikzpicture}
\draw[thick] (0,0) node[xshift=10, yshift=6]{$z$}node[left]{$3$}-- (2,0)node[xshift=-10, yshift=6]{$z'$}node[right]{$0$} -- (1,{2*sin(60)})node[yshift=-12]{$z''$}node[above]{$1$} -- cycle;
\end{tikzpicture}& \begin{tikzpicture}
\draw[thick] (0,0) node[xshift=10, yshift=6]{$z$}node[left]{$2$}-- (2,0)node[xshift=-10, yshift=6]{$z'$}node[right]{$0$} -- (1,{2*sin(60)})node[yshift=-12]{$z''$}node[above]{$1$} -- cycle;
\end{tikzpicture}\\
\end{array}
$$

Let $\sum{x_{\Delta}}s_{\Delta}\in \ker D$, and put a curve $\gamma$ in normal position with respect to the induced triangulation on $\partial M$. The log-holonomy of $\gamma$ is obtained by summing the logarithmic shape parameters of the corners cut off by $\gamma$, with signs determined by the orientation of $\gamma$. See \cite{NZ85,DG13}. As a consequence,
\begin{equation}\label{eqn:cornerpairing}
\begin{aligned}
\langle dH_{\gamma}(\boldsymbol x), \sum_{\Delta\in T}x_{\Delta}s_{\Delta}\rangle&=\langle\sum_{\Delta\in T}\sum_{i=0}^{3}(C_{\Delta, i}d\log z_{\Delta}+C_{\Delta, i}'d\log{z_{\Delta}'}+C_{\Delta, i}''d\log z_{\Delta}'), \sum_{\Delta\in T}x_{\Delta}s_{\Delta}\rangle\\
&=\sum_{\Delta \in T}x_{\Delta}\sum_{i=0}^{3}(C_{\Delta, i}-C_{\Delta, i}')z_{\Delta}''+(C_{\Delta, i}''-C_{\Delta, i}')z_{\Delta}^{-1}.
\end{aligned}
\end{equation}
where $C_{\Delta, i}$, $C_{\Delta, i}'$ and $C_{\Delta, i}''$ count the occurence of the following configuratison at the $i$-the vertex of $\Delta$, respectively:
$$
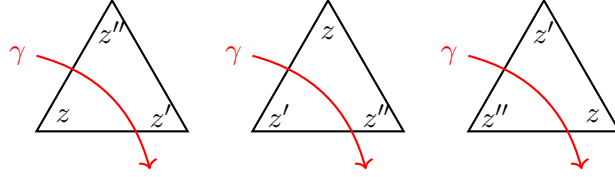
\begin{array}{ccc}
\begin{tikzpicture}
\draw[thick] (0,0) node[xshift=10, yshift=6]{$z$}-- (2,0)node[xshift=-10, yshift=6]{$z'$} -- (1,{2*sin(60)})node[yshift=-12]{$z''$} -- cycle;
 \draw[thick, red, ->](0, 1)node[left]{$\gamma$} to[bend left] (1.5, -0.5);
\end{tikzpicture}  & \begin{tikzpicture}
\draw[thick] (0,0) node[xshift=10, yshift=6]{$z'$}-- (2,0)node[xshift=-10, yshift=6]{$z''$} -- (1,{2*sin(60)})node[yshift=-12]{$z$} -- cycle;
 \draw[thick, red, ->](0, 1)node[left]{$\gamma$} to[bend left] (1.5, -0.5);
\end{tikzpicture}
&
\begin{tikzpicture}
\draw[thick] (0,0) node[xshift=10, yshift=6]{$z''$}-- (2,0)node[xshift=-10, yshift=6]{$z$} -- (1,{2*sin(60)})node[yshift=-12]{$z'$} -- cycle;
 \draw[thick, red, ->](0, 1)node[left]{$\gamma$} to[bend left] (1.5, -0.5);
\end{tikzpicture}\\
\end{array}
$$
 Configurations with the orientation of $\gamma$ reversed contribute $(-1)$ to $C_{\Delta, i}$, $C_{\Delta, i}'$ and $C_{\Delta, i}''$ respectively. The second equality in \eqref{eqn:cornerpairing} follows from the direct computation
$$\langle d\log z_{\Delta}, s_{\Delta} \rangle= \frac{\partial \log z_{\Delta}}{\partial \log z_{\Delta}'}=-z'',$$
$$\langle d\log z_{\Delta}', s_{\Delta} \rangle= \frac{\partial \log z_{\Delta}'}{\partial \log z_{\Delta}'}=1,$$
$$\langle d\log z_{\Delta}'', s_{\Delta} \rangle= \frac{\partial \log z_{\Delta}''}{\partial \log z_{\Delta}'}=-z^{-1}.$$
To prove the proposition, we show the right-hand side of \eqref{eqn:cornerpairing} also equals $[x, \gamma\otimes I_{v}]$. By equation \eqref{eq:expandsdelta}
$$[x, \gamma\times I]=\sum_{\Delta\in T}\sum_{jk\in \{12, 31, 23\}}i(\gamma, f_{0jk}(\Delta))B(\mathbf I_{v}, M_{\mathfrak f}(e_{jk}(\Delta))).$$
The summand could be written in terms of all $12$ configurations, i.e.,
$$\sum_{jk\in \{12, 31, 23\}}i(\gamma, f_{0jk}(\Delta))B(I_{v}, M_{\mathfrak f}(e_{jk}(\Delta)))=\sum_{i=0}^{3}C_{\Delta, i}y_{\Delta, i}+C_{\Delta, i}'y_{\Delta, i}'+C_{\Delta, i}''y_{\Delta, i}''$$
where $y_{\Delta, i}, y_{\Delta, i}'$ and $y_{\Delta, i}''$ compute the contribution of a normal arc at vertex $i$, cutting the corner of $\partial_{i}\Delta$ labeled by $z$, $z'$ and $z''$ respectively. It is straightforward to check
$$y_{\Delta, i}=-\frac{z''}{2}, \quad y_{\Delta, i}'=\frac{1}{2}, \quad y_{\Delta, i}''=-\frac{1}{2z}.$$
We compute $y_{\Delta, 0}$ below as an example, the computation for the  other $11$ coefficients are identical. In this case, the normal arc enters through face $f_{013}(\Delta)$ and exits through face $f_{012}(\Delta)$. Hence
\begin{align}
y_{\Delta, 0}=\frac{1}{2}B(\mathbf I_{v_{0}(\Delta)}, M_{13})-\frac{z^{-1}}{2}B(\mathbf I_{v_{0}(\Delta)}, M_{12})=-\frac{z''}{2}.
\end{align}
\end{proof}

\begin{proposition}\label{prop:pairing2}
Given $y\in \mathrm{coker}(D)$, choose any lift $\widetilde{y}$ of $y$ in $\mathbb C^{E}$. Then the natural bilinear pairing between $\mathrm{coker}(D)\cong H_{1}(\mathcal{A}_{\bullet})$ and $H_{2}(M, \rho)$ satisfies
\begin{equation}\label{eq:pairing2}
[y, [T_{v}]\otimes \mathbf I_{v}]=(C^{T}\tilde{y})_{v}\end{equation}
where $[T_{v}]$ is the fundamental class of $\partial_{v} M$, $I_{v}$ is the invariant vector associated to vertex $v$, and $C$ is the edge-vertex adjacency matrix.
\end{proposition}
\begin{proof}
Since all elements in $\mathcal{A}_{1}$ are cycles. It suffices to verify the equation for $y=\kappa_{e}$. Suppose $I_{v}$ has two eigenvectors $v_{0}$ and $v_{1}$, with $v_{1}=v$. We compute
$$
[\kappa_{e}, [T_{v}]\otimes \mathbf I_{v}]=\sum_{g\in \pi_{1}(M)}i(e, g(T_{v}))B(\begin{bmatrix}
v_{0}(e), v_{1}(e)
\end{bmatrix}\begin{bmatrix}
1, 0\\
0, -1
\end{bmatrix}\begin{bmatrix}
v_{0}(e), v_{1}(e)
\end{bmatrix}^{-1}, \begin{bmatrix}
gv_{0}, gv_{1}
\end{bmatrix}\frac{1}{2}\begin{bmatrix}
1, 0\\
0, -1
\end{bmatrix}\begin{bmatrix}
gv_{0}, gv_{1}
\end{bmatrix}^{-1}),$$
where $i(-, -)$ is the algebraic intersection number of two cycles.
We prove Equation \eqref{eq:pairing2} case by case:
\begin{enumerate}
 \item   If $e$ and $T_{v}$ do not intersect in $M$, the equation is satisfied trivially.
 \item   If $e$ and $T_{v}$ intersect at $v_{0}(e)$ in $M$, then we have the algebraic intersection number equals $1$, and $v_{1}=v_{0}(e)+Cv_{1}(e)$ for some constant $C$. By direct computation, the Bilinear form evaluates to $1$. Hence the equation is satisfied.
 \item   If $e$ and $T_{v}$ intersect at $v_{1}(e)$ in $M$, then we have the algebraic intersection number equals $-1$, and $v_{0}=v_{1}(e)+Cv_{0}(e)$. By direct computation, the Bilinear form evaluates to $-1$. Hence the equation is satisfied.
 \item   If $e$ and $T_{v}$ intersect twice in $M$, then $e$ intersects $T_{v}$ at $v_{0}$ and intersects $gT_{v}$ at $v_{1}$ for some $g$ in the universal cover. By the same computation, each summand contributes $1$, and the equation follows.
\end{enumerate}
\end{proof}

Since $C$ is the edge-vertex adjacency matrix. We have the following corollary immediately.
\begin{corollary}
Let $y_{1}, \ldots, y_{|V|}$ be the basis dual to Porti's basis $\boldsymbol{h}_{M}^{2}$ of $H_{2}(M; \mathrm{Ad}\rho)$. Then for any of their lift $\widetilde{y}_{1},\ldots,\widetilde{y}_{|V|}$ we have
$$C^{T}(\widetilde{y}_{1},\ldots,\widetilde{y}_{|V|})=I,$$
where $C$ is the edge-vertex adjacency matrix.
\end{corollary}

The following linear algebra lemma makes it possible to express the torsion of $\mathbb C^{T}\xrightarrow{D} \mathbb C^{E}\to 0$ by the determinant of a bordered matrix.
\begin{lemma}\label{lem:bordered}
Let $D:\mathbb C^{N}\to \mathbb C^{N}$ be a linear map of rank $N-k$. Let $\boldsymbol{h}_{1}=\{h_{1}^{1},\ldots h_{1}^{k}\}$ and $\boldsymbol{h}_{2}=\{h_{2}^{1},\ldots h_{2}^{k}\}$ be bases of the homology groups of
$$\mathbb C^{N}\to \mathbb C^{N}\to 0.$$
Let $p_{1},\ldots, p_{k}$ be $n$-dimensional row vectors, considered as linear functionals on $\mathbb C^{N}$ such that $[p_{i}(h_{2}^{j})]_{1\le i, j\le k}$ is invertible. Then for any lift $\widetilde{\boldsymbol h_{1}}$ of $\boldsymbol{h}_{1}$ in $\mathbb C^{N}$
$$\mathrm{tor}(C^{N}\xrightarrow{D} C^{N}\to 0 ;\boldsymbol{h}_{1}\cup \boldsymbol{h}_{2})=\pm \frac{1}{\det [p_{i}(h_{2}^{j})]_{1\le i, j\le k}}\begin{bmatrix}
   D &\widetilde{\boldsymbol{h}_{1}}\\
   [p_{i}(e_{j})]_{1\le i\le k, 1\le j\le N}  & 0 \\
\end{bmatrix}.$$
where $\{e_{j}\}$ is the  standard basis of $\mathbb{C}^{N}$.
\end{lemma}
\begin{proof}
Let $U=(u_{1},\ldots, u_{N-k})$ consists column vectors such that $DU$ form a basis of $\mathrm{Im}(D)$. Then, by definition of torsion,
$$\mathrm{tor}=\pm\frac{\det(DU, \boldsymbol h_{2})}{\det(U, \widetilde{\boldsymbol{h}_{1}})}.$$
Observe
$$\det\begin{bmatrix}
D & \boldsymbol h_{2}\\
P & 0
\end{bmatrix}\det\begin{bmatrix}
U & \widetilde{\boldsymbol{h}_{1}} & 0\\
0 & 0 & I_{k}
\end{bmatrix}=\det\begin{bmatrix}
DU & 0 & \boldsymbol{h}_{2}\\
PU & P\widetilde{\boldsymbol{h}_{1}} & 0
\end{bmatrix}=\pm \det\begin{bmatrix}
DU & \boldsymbol{h}_{2} & 0\\
PU & 0 & P\widetilde{\boldsymbol{h}_{1}}
\end{bmatrix}=\det(DU, \boldsymbol{h}_{2})\det(P\boldsymbol{h}_{1}).$$
Moving the terms one obtains
$$\mathrm{tor}=\pm \frac{1}{\det{P\widetilde{\boldsymbol{h}_{1}}}}\det\begin{bmatrix}
D & \boldsymbol h_{2}\\
P & 0
\end{bmatrix}.$$
\end{proof}

\begin{proof}[Proof of theorem \ref{thm:mainincomplete}]
By Lemma \ref{lem:bordered} and Propositions \ref{prop:matrixK} \ref{prop:pairing1} and \ref{prop:pairing2}, the torsion of $\mathcal{K}_{\bullet}^{\boldsymbol{G}}$ with respect to the dual Porti's basis $\boldsymbol{h}_{\bullet}^{\boldsymbol{\gamma}}$ equals to 
$$\mathrm{tor}(\mathbb C^{T}\xrightarrow{K}\mathbb C^{E}\to 0; \boldsymbol{h}_{\bullet}^{\boldsymbol{\gamma}})=\pm \det\begin{bmatrix}
\frac{1}{2}(A\Delta_{z''}+B\Delta_{z}^{-1}) & R\\
\frac{1}{2}(A_{\boldsymbol{\gamma}}\Delta_{z''}+B_{\boldsymbol{\gamma}}\Delta_{z}^{-1}) & 0
\end{bmatrix}.$$
where $R$ is a right inverse of $C^{T}$. Choosing $R=C(C^{T}C)^{-1}$, we have 
$$\mathrm{tor}(\mathcal{K}^{\boldsymbol{G}}; \boldsymbol{h}_{\bullet}^{\boldsymbol \gamma})=\pm \frac{1}{2^{|E|}\det(C^{T}C)}\det\begin{bmatrix}
K & C\\
L_{\boldsymbol{\gamma}} &0
\end{bmatrix}.$$
Apply Proposition \ref{prop:flat}, we have
\begin{align*}
    \mathrm{tor}(\mathcal{K}; \boldsymbol{h}_{\bullet}^{\boldsymbol \gamma})&=\pm\frac{1}{2^{|E|}\det(C^{T}C)}\det\begin{bmatrix}
K & C\\
L_{\boldsymbol{\gamma}} &0
\end{bmatrix}2^{|T|}\prod_{\Delta\in T}z_{\Delta}^{f_{\Delta}''}(z_{\Delta}'')^{-f_{\Delta}}\\
&=\pm\frac{1}{\det(C^{T}C)}\det\begin{bmatrix}
K & C\\
L_{\boldsymbol{\gamma}} &0
\end{bmatrix}\prod_{\Delta\in T}z_{\Delta}^{f_{\Delta}''}(z_{\Delta}'')^{-f_{\Delta}}
\end{align*}
which equals $\mathrm{tor}(M, \partial M; \boldsymbol{h}_{\bullet}^{\boldsymbol \gamma})$ by equation \eqref{eq:applyM}. This is exactly Theorem \ref{thm:mainincomplete}.
\end{proof}

\begin{proof}[Proof of Theorem \ref{thm:main1}]
Let $f_{\Delta}, f_{\Delta}', f_{\Delta}''$ be a strong flattening with respect to $\mathcal{T}$. By Theorem \ref{thm:mainincomplete}, we have verified equation \eqref{eq:main1} for characters $\rho$ such that the induced character on each boundary torus is nontrivial. Since this is a generic condition, we have verified equation \eqref{eq:main1} on a dense subset of $\boldsymbol \gamma$-regular locus of $\mathrm X^{|V|}(M)\cap \mathrm  X^{irr}(M)$. By Theorem \ref{thm:rationalfunction}, both side of $\eqref{eq:main1}$ are rational functions on the Zariski open set; hence equation \eqref{eq:main1} holds for every characters in the Zariski open set. 

For the distinguished character $\rho_{0}$, by \cite[Section 3.5]{DG13}, the value of the flattening monomials does not depend on the combinatorial flattenings. Hence \eqref{eq:main1} holds for all combinatorial flattenings for the distinguished character $\rho_{0}$. This completes the proof of \ref{thm:main1}.
\end{proof}

\section{Proof of Theorem \ref{thm:main2}}
In this section, we prove Theorem \ref{thm:main2}. We first rewrite the one-loop formula \eqref{eq:main1} as follows. Let $P$ be a matrix of size $(|E|-|V|)\times |E|$ whose rows are the standard basis vectors $\{\boldsymbol{e}_{a_{1}},\ldots e_{a_{|E|-|V|}}\}$, such that
$$\mathrm{rank}(P(A\Delta_{z''}+B\Delta_{z}^{-1}))=\mathrm{rank}(PK)=|E|-|V|.$$
Now we define $A^{red}=PA$, $B^{red}=PB$ and $K^{red}=PK$. In other words, $A^{red}$, $B^{red}$ and $K^{red}$ can be obtained from $A$ and $B$ by removing $|V|$ rows. The remaining rows correspond to the subset $E^{red}:=\{e_{a_{1},\ldots e_{a_{|E|-|V|}}}\}$. Now the torsion of $M$ can be written as
$$\mathrm{tor}(M, \boldsymbol m)=\pm \frac{1}{\det(P^{T}, C)}\det\begin{pmatrix}
K^{red}\\
\frac{dH(m_{i})}{d\log z_{j}'}
\end{pmatrix}\prod_{j=1}^{N}z_{j}^{f_{j}''}(z_{j}'')^{-f_{j}}.$$
\begin{remark}
The form above is exactly the one-loop formula appears in \cite{DG13}.
\end{remark}

We denote the log-holonomy of a curve around an edge $e_{i}$ by $H(e_{i})$, so $K^{red}=[\frac{dH(e_{a_k})}{d\log(z_{j})}]$. The matrix 
$$\begin{pmatrix}
K^{red} \\
L_{\boldsymbol{m}}
\end{pmatrix}$$
could be realized as the Jacobian matrix of a holomorphic map sending 
$$\log z_{1}',\ldots,\log z_{|T|}'$$ to 
$$H(e_{a_1}),\ldots, H(e_{a_{|E|-|V|}}), H(m_{1}),\ldots, H(m_{|V|})$$
at the distinguished character $\rho_{0}$. Using the theorem below, we obtain the real Jacobian of the same map.
\begin{theorem}\label{thm:real-complexJ}
Let $F: \mathbb{C}^{n}\to \mathbb C^{n}$ be a holomorphic function that is non-degenerate at $z_{0}$. Define its real coordinate by $z_{j}=x_{j}+i y_{j}$. Then
$$\det_{\mathbb R}F|_{z=z_{0}}=\left|\det_{\mathbb{C}}F|_{z=z_{0}}\right|^{2}.$$
\end{theorem}

\begin{corollary}\label{cor:jacobian}
Viewing $T_{\rho_{0}}^{*}\mathcal{Z}(\mathcal{T})$ as a $2|T|$-dimensional real vector space, we have
$$\det\left(\frac{d\mathrm{Re}H(e_{a_k}), d\mathrm{Im}H(e_{a_k}), d\mathrm{Re}H(m_{i}), d\mathrm{Im}H(m_{i})}{d\log |z_{j}'|, d\arg z_{j}'}\right)=\pm \left|\det\begin{pmatrix}
    K^{red} \\
    \frac{dH(m_{i})}{d\log z_{j}'}
\end{pmatrix}\right|^{2}$$
\end{corollary}

Notice that at $\rho_{0}$, $\mathrm{Re} H(m_{i})$ and $\mathrm{Re} H(l_{i})$ form a system of local coordinates of the deformation space. By Cauchy-Riemann equation, we have
\begin{equation}\label{eqn:realimaginaryhol}
\det \left(\frac{d\mathrm{Re}H(m_{i}), d\mathrm{Re}H(l_{j})}{d\mathrm{Re} H(m_{j}), d\mathrm{Im}H(m_{i})}\right)=(-1)^{|V|}\det\left(\mathrm{Im}\frac{dH(l_{i})}{dH(m_{j})}\right).\end{equation}

\begin{lemma}\label{lem:localjac}
Suppose $z$ takes value in the upper half complex plane, and viewing $z'=\frac{1}{1-z}$ and $z''=1-\frac{1}{z}$ as functions of $z$. Then we have
$$\det\frac{(d\log |z|, d\log |z''|)}{(d\log z', d\arg z')}=\frac{\mathrm{Im}(z)}{|z|^{2}}$$
\end{lemma}
\begin{proof}
By direct computations.
\end{proof}
Applying Lemma \ref{lem:localjac} and Equation \eqref{eqn:realimaginaryhol}, We obtain, up to sign,
\begin{equation}\label{eqn:jacobirelation}
\det\left(\frac{d\theta_{a_{k}}, d\mathrm{Re}H(e_{a_k}), d\mathrm{Re}H(m_{i}), d\mathrm{Re}H(l_{i})}{d\log |z_{j}|, d\log |z_{j}''|}\right)=\det\left(\mathrm{Im}\frac{dH(l_{i})}{dH(m_{j})}\right)\prod_{i\in T}\frac{|z_{i}|^{2}}{\mathrm{Im}(z_{i})}\left|\det\begin{pmatrix}
    K^{red} \\
    \frac{dH(m_{i})}{d\log z_{j}'}
\end{pmatrix}\right|^{2}
\end{equation}
To compute the left-hand side, we utilize the symplectic nature of Neumann-Zagier matrices.

Consider the linear map $\mathbb{R}^{E}\rightarrow  \mathbb R^{2T}$ sending decorated edge length to the space of tetrahedron shapes with coordinates $(\log |z_{i}|, \log|z_{i}''|)_{i\in T}$.
\begin{lemma}\label{lem:lengthtoshape}
The matrix representation of the linear map $\mathbb{R}^{E}\rightarrow  \mathbb R^{2T}$ is
$$\frac{1}{2}\begin{pmatrix}
-B^{T} \\
A^{T}
\end{pmatrix}$$
In addition, its restriction to $\mathbb R^{E^{red}}$ has the same image, and is represented by 
$S=\frac{1}{2}(-B^{red}, A^{red})^{T}$.
\end{lemma}
\begin{proof}
Denote the sum of the decorated edge lengths of pairs of opposite edges in a tetrahedron $\Delta\in T$ by $b_{\Delta}, b_{\Delta}'$ and $b_{\Delta}''$. By Equation \eqref{eqn:zinlength}
\begin{equation}\label{eqn:lengthlogz}
\log |z|=\frac{b_{\Delta}'-b_{\Delta}''}{2}, \hspace{5mm}\log |z''|=\frac{b_{\Delta}-b_{\Delta}'}{2}.\end{equation}
On the other hand, the quantity $b_{\Delta}, b_{\Delta}'$ and $b_{\Delta}''$ and decorated edge lengths $x_{e}$ are related by the Neumann-Zagier matrices:
$$b_{\Delta}=(G^{T}\boldsymbol x)_{\Delta}, \hspace{3mm}b_{\Delta}'=(G'^{T}\boldsymbol x)_{\Delta}, \hspace{3mm} b_{\Delta}''=(G''^{T}\boldsymbol x)_{\Delta},$$
The first statement follows. The second statement follows because $S$ contains exactly the linearly independent columns.
\end{proof}

Let $(S, R)$ be a basis of $T_{\rho}\mathbb{R}^{2T}$, such that
\begin{enumerate}
    \item $S=\frac{1}{2}(-B^{red}, A^{red})^{T}$, i.e., their columns lies in the image of $(\mathbb R^{E}\to \mathbb R^{2T})$.
    \item $R$ satisfies
    $$\begin{pmatrix}
     A^{red}& B^{red}\\
     A_{\boldsymbol m} & B_{\boldsymbol m}\\
     A_{\boldsymbol l} & B_{\boldsymbol l}
    \end{pmatrix}R=I$$
\end{enumerate}
By the defining property of the matrix $R$, we have
\begin{lemma}\label{lem:changetangentbase}
\begin{align*}
\det\left(\frac{d\theta_{a_{k}}, d\mathrm{Re}H(e_{a_k}), d\mathrm{Re}H(m_{i}), d\mathrm{Re}H(l_{i})}{d\log |z_{j}|, d\log |z_{j}''|}\right)&=\det(\langle d\theta_{a_{k}}, d\mathrm{Re}H(e_{a_k}), d\mathrm{Re}H(m_{i}), d\mathrm{Re}H(l_{i}); \partial_{\log |z_{j}|}, \partial_{\log|z_{j}''|}\rangle)\\
&=\det\begin{pmatrix}
\frac{d\boldsymbol{\theta}^{red}}{d\boldsymbol{x}^{red}}&* &* &*\\
0&I&0&0\\
0&0&I&0\\
0&0&0&I
\end{pmatrix}\frac{1}{\det(S, R)}\\
&=\det\left(\frac{d\boldsymbol{\theta}^{red}}{d\boldsymbol{x}^{red}}\right)\frac{1}{\det(S, R)}
\end{align*}
\end{lemma}
\begin{proof}
The first equality is a reinterpretation of the Jacobian matrix in terms of the pairing between tangent vectors and cotangent vectors. The second equality comes from changing the standard basis of $\mathbb R^{2T}$ to $S, R$.
\end{proof}
To compute $\det(S, R)$, we utilize the following symplectic linear algebra lemma.
\begin{lemma}\label{lem:symplecticlinear}
Consider $\mathbb R^{2N}$ equipped with the standard symplectic structure
$$\Omega_{N}:=\begin{bmatrix}
    0 &I_{N} \\
    -I_{N} & 0
\end{bmatrix}.$$
Suppose $X$ consists of linearly independent column vectors that span an isotropic subspace. Let $\Phi=[\Phi_{0},\Phi_{1}]$ be a $2N \times (m+(2N-2m))$-matrix satisfying
\begin{enumerate}
    \item $\Phi^{T} X=0$,
    \item $\Phi_{0}^{T}\Omega_{N}\Phi_{1}=0$, and
    \item $\Phi_{1}^{T}\Omega_{N} \Phi_{1}=\Omega_{N-m}$.
\end{enumerate}
Then for any matrix $Y$ such that $\Phi^{T}Y=I$, one has
$$|\det(X,Y)|=|\frac{1}{\det(-\Phi_{0}^{T}\Omega_{N}X(X^{T}X)^{-1})}|$$
\end{lemma}
\begin{proof}
Observe
$$\begin{pmatrix}
(X^{T}X)^{-1}X^{T} \\
\Phi^{T}
\end{pmatrix}\begin{pmatrix}
X & Y
\end{pmatrix}=\begin{pmatrix}
I & *\\
\Phi^{T}X & I
\end{pmatrix}=\begin{pmatrix}
I & *\\
0 & I
\end{pmatrix}.$$
Hence it suffices to compute the determinant of 
$$Z:=\begin{pmatrix}
(X^{T}X)^{-1}X^{T} \\
\Phi^{T}
\end{pmatrix}$$
Also observe
$$Z\Omega_{N} Z^{T}=\begin{pmatrix}
* & (X^{T}X)^{-1}X^{T}\Omega_{N} \Phi_{0} & *\\
\Phi_{0}^{T}\Omega_{N} X(X^{T}X)^{-1} &0 &0 \\
* & 0&\Omega_{N-m}
\end{pmatrix}$$
Therefore, we compute
$|\det(Z)|=\sqrt{|\det(Z\Omega_{N}Z^{T})|}=\det(\Phi_{0}^{T}\Omega_{N} X(X^{T}X)^{-1})$
\end{proof}

\begin{corollary}
$$|\det(S, R)|=\frac{1}{2^{|E|}}.$$
\end{corollary}
\begin{proof}
Use $X=S$, $Y=\sqrt{2}R$, $\Phi_{0}=\frac{1}{\sqrt{2}}(A^{red}, B^{red})$, 
$$\Phi_{1}=\frac{1}{\sqrt{2}}\begin{pmatrix}
A_{\boldsymbol{m}} &B_{\boldsymbol{m}}\\
A_{\boldsymbol{l}} &B_{\boldsymbol{l}}
\end{pmatrix}$$
and the symplectic form $\Omega=\sum_{j\in T}d\log |z_{j}|\wedge d\log |z_{j}''|$. With these choices, Proposition \ref{prop:NZsymplectic} implies that the conditions of Lemma \ref{lem:symplecticlinear} is satisfied. Applying the lemma, we have
\begin{align*}
|\det(S, \sqrt{2}R)|&=\frac{\sqrt{2}^{|E|-|V|}2^{-|E|+|V|}}{\det\left((A^{red}, B^{red})\begin{pmatrix}
0 & I\\
-I & 0 \\
\end{pmatrix}\begin{pmatrix}
(-B^{red})^{T}\\
(A^{red})^{T}
\end{pmatrix}\left((A^{red}, B^{red})\begin{pmatrix}
(A^{red})^{T}\\
(B^{red})^{T}
\end{pmatrix}\right)^{-1}\right)}\\
&=\sqrt{2}^{-|E|+|V|}
\end{align*}
Since $R$ contains $|V|+|E|$ columns, we have 
$$\det(S, R)=\frac{1}{\sqrt{2}^{|E|+|V|}}\det(S, \sqrt{2}R)=\frac{1}{2^{|E|}}$$
\end{proof}

Combining Lemma \ref{lem:changetangentbase} and Equation \eqref{eqn:jacobirelation}, we obtains the following formula:
\begin{equation}\label{eq:holomorphicJac}
\det\left(\frac{d\boldsymbol{\theta}^{red}}{d\boldsymbol{x}^{red}}\right)2^{|E|}=\det\left(\mathrm{Im}\frac{dH(l_{i})}{dH(m_{j})}\right)\prod_{i\in T}\frac{|z_{i}|^{2}}{\mathrm{Im}(z_{i})}\left|\det\begin{pmatrix}
    K^{red} \\
    L_{\boldsymbol{m}}
\end{pmatrix}\right|^{2}.
\end{equation}
We still need one final ingredient relating the determianants of gram matrices to the flattening terms in \eqref{eq:main1}.
\begin{proposition}\label{prop:graminshape}
Let $f_{\Delta}, f_{\Delta}', f_{\Delta}''$ be a combinatorial flattening, we have
\begin{equation}\label{eq:realflat}
    e^{-\sum_{e}2x_{e}}\prod_{\Delta\in T}2\sqrt{-\det(\mathrm{Gram}(\Delta))}=\prod_{\Delta\in T}\frac{\mathrm{Im}(z_{\Delta})}{|z_{\Delta}|^{2}}|z_{\Delta}|^{2f_{\Delta}''}|z_{\Delta}''|^{-2f_{\Delta}}\end{equation}
\end{proposition}
\begin{proof} By the definition of combinatorial flattening, we have
$$\sum_{e\in E}2x_{e}=\sum_{\Delta\in T}s_{\Delta}f_{\Delta}+s_{\Delta}'f_{\Delta}'+s_{\Delta}''f_{\Delta}''.$$
Using Proposition \ref{prop:grambyz} and Equation \eqref{eqn:zinlength}, we obtain
$$\sqrt{-\det(\mathrm{Gram}(\Delta))}=\frac{1}{2}e^{s'_{\Delta}}\frac{\mathrm{Im}(z_{\Delta})}{|z_{\Delta}|^{2}}.$$
Substituting into the left-hand side of \eqref{eq:realflat}. One find it suffices to prove
$$e^{-b_{\Delta}f_{\Delta}-b_{\Delta}'f_{\Delta}'-b_{\Delta}''f_{\Delta}''}e^{b_{\Delta}'}=|z_{\Delta}|^{2f_{\Delta}''}|z_{\Delta}''|^{-2f_{\Delta}},$$
which follows by rewriting $f_{\Delta}'=1-f_{\Delta}-f_{\Delta}''$, and apply Equation \eqref{eqn:lengthlogz}.
\end{proof}
\begin{proof}[Proof of Theorem \ref{thm:main2}]
Combining Proposition \ref{prop:graminshape} and Equation \eqref{eq:holomorphicJac}, we obtains
\begin{equation}\label{eq:HJplusFL}\det\left(\frac{d\boldsymbol{\theta}^{red}}{d\boldsymbol{x}^{red}}\right)e^{-\sum_{e}2x_{e}}\prod_{\Delta\in T}\sqrt{-\mathrm{Gram}(\Delta)}=\det\left(\mathrm{Im}\frac{dH(l_{i})}{dH(m_{j})}\right)\det(P^{T}, C)^{2}|\mathrm{tor}(M, \boldsymbol{m}, \rho_{0})^2|\end{equation}
We apply the following linear algebra lemma to recover the determinant of the full Hessian $\left(\frac{d\boldsymbol{\theta}}{d\boldsymbol{x}}\right)$ from the reduced one.
\begin{lemma}\label{lem:symmetricboarder}
Let $J$ be an $N\times N$ symmetric matrix of rank $N-k$. Suppose $C$ is an $N\times k$ matrix such that $JC=0$ and let $P$ be an $(N-k)\times N-k$ matrix consisting of column vectors complementary to those of $C$. Then we have
$$\frac{\det(PJP^{T})}{\det(P^{T}, C)^{2}}=\frac{(-1)^{k}}{\det(C^{T}C)^2}\det\begin{pmatrix}
J & C\\
C^{T} & 0
\end{pmatrix}.$$
\end{lemma}
Here we let $P$ be the matrix that selects the linearly independent rows of $K$, $C$ be the edge-vertex adjacency matrix, and $J$ be the full Jacobian matrix $\left(\frac{d\boldsymbol{\theta}}{d\boldsymbol{x}}\right)$. In \cite[proof of Proposition 6.2]{LMSWY26}, it is verified there that $J$ is a Hessian matrix and hence symmetric, and Lemma \ref{lem:symmetricboarder} is applied. Therefore, we have
$$PJP^{T}=\frac{d\boldsymbol{\theta}^{red}}{d\boldsymbol{x}^{red}}, \quad\text{ and}\quad \frac{\det(\frac{d\boldsymbol{\theta}^{red}}{d\boldsymbol{x}^{red}})}{\det(P^{T}, C)^{2}}=\frac{(-1)^{k}}{\det(C^{T}C)^2}\det\begin{pmatrix}
\left(\frac{d\boldsymbol{\theta}}{d\boldsymbol{x}}\right) & C\\
C^{T} & 0
\end{pmatrix}.$$
Substituting into Equation \eqref{eq:HJplusFL}, we obtain Equation \eqref{eq:main2}. It completes the proof of Theorem \ref{thm:main2} 
\end{proof}

\bibliographystyle{alpha}
\bibliography{a}

@article{Sch78,
  author = {Schwarz, A. S.},
  title = {The partition function of degenerate quadratic functional and {R}ay--{S}inger invariants},
  journal = {Letters in Mathematical Physics},
  year = {1978},
  volume = {2},
  pages = {247--252},
}

@article{Sch79,
  author = {Schwarz, A. S.},
  title = {The partition function of a degenerate functional},
  journal = {Communications in Mathematical Physics},
  year = {1979},
  volume = {67},
  number = {1},
  pages = {1--16},
}

@article{Wit89,
  author = {Witten, Edward},
  title = {Quantum field theory and the {J}ones polynomial},
  journal = {Communications in Mathematical Physics},
  year = {1989},
  volume = {121},
  number = {3},
  pages = {351--399},
}

@article{RT91,
  author = {Reshetikhin, N. and Turaev, V. G.},
  title = {Invariants of {$3$}-manifolds via link polynomials and quantum groups},
  journal = {Inventiones Mathematicae},
  year = {1991},
  volume = {103},
  pages = {547--597},
}

@article{TV92,
  author = {Turaev, V. G. and Viro, O. Y.},
  title = {State sum invariants of {$3$}-manifolds and quantum {$6j$}-symbols},
  journal = {Topology},
  year = {1992},
  volume = {31},
  number = {4},
  pages = {865--902},
}

@article{Kas95,
  author = {Kashaev, R. M.},
  title = {A link invariant from quantum dilogarithm},
  journal = {Modern Physics Letters A},
  year = {1995},
  volume = {10},
  number = {19},
  pages = {1409--1418},
}

@article{CY18,
  author = {Chen, Qingtao and Yang, Tian},
  title = {Volume conjectures for the {R}eshetikhin--{T}uraev and the {T}uraev--{V}iro invariants},
  journal = {Quantum Topology},
  year = {2018},
  volume = {9},
  number = {3},
  pages = {419--460},
}

@article{DG13,
  author = {Dimofte, Tudor and Garoufalidis, Stavros},
  title = {The quantum content of the gluing equations},
  journal = {Geometry \& Topology},
  year = {2013},
  volume = {17},
  number = {3},
  pages = {1253--1315},
}

@article{Sie21,
  author = {Siejakowski, Rafa{\l}},
  title = {Infinitesimal gluing equations and the adjoint hyperbolic {R}eidemeister torsion},
  journal = {Tohoku Mathematical Journal},
  year = {2021},
  volume = {73},
  number = {4},
  eprint = {1710.02109},
  archivePrefix = {arXiv},
}

@article{Yoo24,
  author = {Yoon, Seokbeom},
  title = {The twisted {$1$}-loop invariant and the {J}acobian of {P}tolemy coordinates},
  journal = {Mathematische Zeitschrift},
  year = {2024},
  volume = {307},
  number = {1},
  note = {Paper No.~19},
}

@misc{DGY23,
  author = {Dunfield, Nathan M. and Garoufalidis, Stavros and Yoon, Seokbeom},
  title = {{$1$}-loop equals torsion for fibered {$3$}-manifolds},
  year = {2023},
  note = {Preprint, arXiv:2304.00469},
  eprint = {2304.00469},
  archivePrefix = {arXiv},
}

@misc{PW23,
  author = {Pandey, Tushar and Wong, Ka Ho},
  title = {Geometry of fundamental shadow link complements and applications to the {$1$}-loop conjecture},
  year = {2023},
  note = {Preprint, arXiv:2308.06643},
  eprint = {2308.06643},
  archivePrefix = {arXiv},
}

@article{GY26,
  author = {Garoufalidis, Stavros and Yoon, Seokbeom},
  title = {{$1$}-loop equals torsion for two-bridge knots},
  journal = {Quantum Topology},
  year = {2026},
  note = {Published online first, 10 June 2026},
}

@article{CM23,
  author = {Chen, Qingtao and Murakami, Jun},
  title = {Asymptotics of quantum {$6j$} symbols},
  journal = {Journal of Differential Geometry},
  year = {2023},
  volume = {123},
  number = {1},
  pages = {1--20},
}

@article{WY24,
  author = {Wong, Ka Ho and Yang, Tian},
  title = {Adjoint twisted {R}eidemeister torsion and {G}ram matrices},
  journal = {Advances in Mathematics},
  year = {2024},
  volume = {438},
  note = {Paper No.~109470},
}

@misc{LMSWY25a,
  author = {Liu, Tianyue and Ming, Shuang and Sun, Xin and Wu, Baojun and Yang, Tian},
  title = {{T}uraev--{V}iro invariant from the modular double of {$\mathrm{U}_{q}\mathfrak{sl}(2;\mathbb{R})$}},
  year = {2025},
  note = {Preprint, arXiv:2508.05120},
  eprint = {2508.05120},
  archivePrefix = {arXiv},
}

@misc{LMSWY25b,
  author = {Liu, Tianyue and Ming, Shuang and Sun, Xin and Wu, Baojun and Yang, Tian},
  title = {Asymptotics of {$b$}-{$6j$} symbols and anti-de {S}itter tetrahedra},
  year = {2025},
  note = {Preprint, arXiv:2511.20953},
  eprint = {2511.20953},
  archivePrefix = {arXiv},
}

@misc{LMSWY26,
  author = {Liu, Tianyue and Ming, Shuang and Sun, Xin and Wu, Baojun and Yang, Tian},
  title = {{$\mathrm{U}_{q\widetilde q}\mathfrak{sl}(2;\mathbb{R})$} {T}uraev--{V}iro invariants for cusped {$3$}-manifolds},
  year = {2026},
  note = {Preprint, arXiv:2608.16560},
  eprint = {2608.16560},
  archivePrefix = {arXiv},
}

@misc{HKS21,
  author = {Hodgson, Craig D. and Kricker, Andrew J. and Siejakowski, Rafa{\l} M.},
  title = {On the asymptotics of the meromorphic {$3$D}-index},
  year = {2021},
  note = {Preprint, arXiv:2109.05355},
  eprint = {2109.05355},
  archivePrefix = {arXiv},
}

@article{Mil62,
  author = {Milnor, John},
  title = {A duality theorem for {R}eidemeister torsion},
  journal = {Annals of Mathematics},
  year = {1962},
  volume = {76},
  number = {1},
  pages = {137--147},
}

@article{Mil66,
  author = {Milnor, John},
  title = {{W}hitehead torsion},
  journal = {Bulletin of the American Mathematical Society},
  year = {1966},
  volume = {72},
  pages = {358--426},
}

@book{Tur01,
  author = {Turaev, Vladimir},
  title = {Introduction to Combinatorial Torsions},
  series = {Lectures in Mathematics {ETH Z\"urich}},
  publisher = {Birkh{\"a}user},
  address = {Basel},
  year = {2001},
}

@article{Por97,
  author = {Porti, Joan},
  title = {Torsion de {R}eidemeister pour les vari{\'e}t{\'e}s hyperboliques},
  journal = {Memoirs of the American Mathematical Society},
  year = {1997},
  volume = {128},
  number = {612},
  note = {x+139 pp.},
}

@incollection{Por18,
  author = {Porti, Joan},
  title = {{R}eidemeister torsion, hyperbolic three-manifolds, and character varieties},
  booktitle = {Handbook of Group Actions. Vol.~IV},
  series = {Advanced Lectures in Mathematics},
  volume = {41},
  publisher = {International Press},
  year = {2018},
  pages = {447--507},
}

@article{Bar07,
  author = {Bar-Natan, Dror},
  title = {Fast {K}hovanov homology computations},
  journal = {Journal of Knot Theory and Its Ramifications},
  year = {2007},
  volume = {16},
  number = {3},
  pages = {243--255},
}

@article{NZ85,
  author = {Neumann, Walter D. and Zagier, Don},
  title = {Volumes of hyperbolic three-manifolds},
  journal = {Topology},
  year = {1985},
  volume = {24},
  number = {3},
  pages = {307--332},
}

@incollection{Neu92,
  author = {Neumann, Walter D.},
  title = {Combinatorics of triangulations and the {C}hern--{S}imons invariant for hyperbolic {$3$}-manifolds},
  booktitle = {Topology '90},
  editor = {Apanasov, Boris and Neumann, Walter D. and Reid, Alan W. and Siebenmann, Laurent},
  series = {Ohio State University Mathematics Research Institute Publications},
  number = {1},
  publisher = {de Gruyter},
  address = {Berlin},
  year = {1992},
  pages = {243--271},
}

@article{PT01,
  title     = {Clebsch--Gordan and Racah--Wigner Coefficients for a Continuous
               Series of Representations of $U_q(\mathfrak{sl}(2,\mathbb{R}))$},
  volume    = {224},
  issn      = {0010-3616},
  url       = {http://dx.doi.org/10.1007/PL00005590},
  doi       = {10.1007/PL00005590},
  number    = {3},
  journal   = {Communications in Mathematical Physics},
  publisher = {Springer Science and Business Media LLC},
  author    = {Ponsot, B. and Teschner, J.},
  year      = {2001},
  month     = dec,
  pages     = {613--655}
}

@misc{PT99,
      title={Liouville bootstrap via harmonic analysis on a noncompact quantum group}, 
      author={B. Ponsot and J. Teschner},
      year={1999},
      eprint={hep-th/9911110},
      archivePrefix={arXiv},
      primaryClass={hep-th},
      url={https://arxiv.org/abs/hep-th/9911110}, 
}
\end{document}